\documentclass[10pt, reqno]{amsart}
\usepackage{amsmath,amssymb,amsthm}

\usepackage{amsmath, amssymb}
\usepackage{amsthm, amsfonts, mathrsfs}
\usepackage{fullpage}
\usepackage{amsfonts,graphicx}

\numberwithin{equation}{section}
\usepackage[colorlinks=true, pdfstartview=FitV, linkcolor=blue, citecolor=blue, urlcolor=blue]{hyperref}

\newtheorem{Theo}{Theorem}[section]
\newtheorem{lem}[Theo]{Lemma}
\newtheorem{cor}[Theo]{Corollary}
\newtheorem{prop}[Theo]{Proposition}

\theoremstyle{plain}
\theoremstyle{definition}
\newtheorem{defi}[Theo]{Definition}

\theoremstyle{remark}

\newtheorem*{rema*}{Remark}

\newcommand{\ZZ}{\mathbb{Z}}

\newcommand{\NN}{\mathbb{N}}

\newcommand{\RR}{\mathbb{R}}

\newcommand{\EE}{\varepsilon}

\newcommand{\Lip}{\textnormal{Lip}}

\newcommand{\Div}{\textnormal{div}}

\newcommand{\DD}{\textnormal{D}}
\newcommand{\loc}{\textnormal{loc}}
\newcommand{\supp}{\textnormal{supp }}

\author[H. Meddour]{Halima Meddour}
\address{LEDPA, Universit\'e de Batna --2--\\ Facult\'e des Math\'ematiques et Informatique\\ D\'epartement de Math\'ematiques\\ 05000 Batna Alg\'erie}
\email{h.meddour@univ-batna2.dz}

\author[M. Zerguine]{Mohamed Zerguine}
\address{LEDPA, Universit\'e de Batna --2--\\ Facult\'e des Math\'ematiques et Informatique\\ D\'epartement de Math\'ematiques\\ 05000 Batna Alg\'erie}
\email{m.zerguine@univ-batna2.dz}
\email{mohamed\_zerguine@yahoo.fr}

\date{}

\begin{document}

\title[Inviscid limit]
{Optimal rate of convergence in Stratified Boussinesq system}
\maketitle
\begin{abstract}
We study the  vortex patch problem for $2d-$stratified Navier-Stokes system.  We aim at extending several results obtained in \cite{ad,danchinpoche,hmidipoche} for standard Euler and Navier-Stokes systems. We shall deal with smooth initial patches and establish global strong estimates uniformly with respect to the viscosity in the spirit of \cite{HZ-poche, Z-poche}. This allows to prove the convergence of the viscous solutions towards the inviscid one. In the setting of a Rankine vortex, we show that the rate of convergence for the vortices is optimal  in $L^p$ space  and is given by $(\mu t)^{\frac{1}{2p}}$. This generalizes the result of \cite{ad} obtained for $L^2$ space.
\end{abstract}
\tableofcontents
\section{Introduction} 
\quad\quad This paper is mainly motivated by the analysis of the initial value problem for the stratified Navier-Stokes system. This system of partial differential equations governs the evolution of a viscous incompressible fluid like the atmosphere and the ocean where one should take into account  the  friction forces and the stratification under the Boussinesq approximation, see \cite{Ped}. The state of the fluid is described by a triplet $(v_\mu,p_\mu,\rho_\mu)$ where  $v_{\mu}(t,x)$ denotes the velocity field which is assumed to be incompressible and the  thermodynamical variables $p_{\mu}(t,x)$ and $\rho_{\mu}(t,x)$ which are two scalar functions representing respectively  the pressure and the density. The equations being solved take the form
\begin{equation}
\left\{ \begin{array}{ll} 
\partial_{t}v_{\mu}+v_{\mu}\cdot\nabla v_{\mu}-\mu\Delta v_{\mu}+\nabla p_{\mu}=\rho_{\mu} \vec e_2 & \textrm{if $(t,x)\in \RR_+\times\RR^2$,}\\
\partial_{t}\rho_{\mu}+v_{\mu}\cdot\nabla \rho_{\mu}-\kappa\Delta \rho_{\mu}=0 & \textrm{if $(t,x)\in \RR_+\times\RR^2$,}\\ 
\Div v_{\mu}=0, &\\ 
({v}_{\mu},{\rho}_{\mu})_{| t=0}=({v}^0_{\mu},{\rho}^0_{\mu}).  
\end{array} \right.\tag{B$_{\mu,\kappa}$}
\end{equation} 
 The two coefficients $\mu, \kappa$ stand respectively for the kinematic viscosity and molecular diffusivity and $\vec e_2=(0,1)$. For a better understanding of the system (B$_{\mu,\kappa}$) it is more convenient to write it  using the  vorticity-density formulation. Thus  the vorticity $\omega\triangleq\partial_1 v^2-\partial_2 v^1$ and the density satisfy the equivalent system,
\begin{equation}
\left\{ \begin{array}{ll} 
\partial_{t}\omega_{\mu}+v_{\mu}\cdot\nabla\omega_{\mu}-\mu\Delta\omega_{\mu}=\partial_1\rho_{\mu} & \textrm{if $(t,x)\in \RR_+\times\RR^2$,}\\
\partial_{t}\rho_{\mu}+v_{\mu}\cdot\nabla\rho_{\mu}-\kappa\Delta\rho_{\mu}=0& \textrm{if $(t,x)\in \RR_+\times\RR^2$,}\\
v_{\mu}=\nabla^\perp\Delta^{-1}\omega_{\mu},\\
(\rho_{\mu},\omega_{\mu})_{| t=0}=(\rho_{\mu}^0,\omega_{\mu}^0).  
\end{array} \right. \tag{VD$_{\mu,\kappa}$}
\end{equation}
It is clear that (B$_{\mu,\kappa}$) coincides with the classical incompressible Navier-Stokes system when the initial density $\rho_{\mu}^0$ is identically constant. For a general  review on the mathematical theory of the Navier-Stokes system we refer  for instance to  \cite{HCD, LemarBook}.  We notice that the system  (B$_{\mu,\kappa}$) is the subject of intensive research activities especially in the last decades.  A lot of results have been obtained and we shall restrict the discussion to some of them. When the coefficients  $\mu$ and $\kappa$ are strictly positive,  it was proved in \cite{CanDib, Guo} that the system (B$_{\mu,\kappa}$) admits a unique global solution  for arbitrarily large data. For $\mu>0, \kappa=0$ the global well-posedness problem was solved independently by Hou and Li  \cite{HL} and Chae \cite{Chae}  for  smooth initial data in Sobolev spaces $H^s, s>2$. Those results were improved by  Abidi and Hmidi in  \cite{ah}  for  $(v^0, \rho^0)\in B^{-1}_{\infty,1}\cap L^2\times B^{0}_{2,1}$. Later, Danchin and Paicu investigated in \cite{dp} the global well-posedness for any initial data $(v^0, \rho^0)$ in $L^2\times L^2$. The opposite case $\mu=0$  and $\kappa>0$ is also well-explored.  Actually, Chae  proved  in \cite{Chae} the global well-posedness for $(v^0, \rho^0)\in H^s\times H^s$ for $s>2$ which was later improved  by Hmidi and Keraani in  \cite{hk1} for critical Besov spaces, that is, $(v^0,\rho^0)\in B^{\frac2p+1}_{p,1}\times B^{-1+\frac{2}{p}}_{p,1}\cap L^r,\;r>2$. The global existence in  the framework of Yudovich solutions was accomplished in \cite{DP}  by Danchin and Paicu for  $(v^0, \rho^0)\in L^2\times L^2\cap B_{\infty, 1}^{-1}$ and $\omega^0\in L^r\cap L^\infty$ with $r\ge 2$. For other connected topics we refer the reader  to  \cite{hk1, HKR1, HKR2, HZ, Tit, MX, ES}.

\quad\quad The main focus of the current paper is twofold. In the first part, we study the persistence regularity of  the vortex patches for (B$_{\mu,\kappa}$) for $\kappa=1$,  denoted  simply by (B$_{\mu}$). In the second part we shall deal with the strong convergence towards the limit system when the viscosity $\mu$ goes to zero. The  limit system  is nothing but the stratified Euler equations,
\begin{equation}
\left\{ \begin{array}{ll} 
\partial_{t}v+v\cdot\nabla v+\nabla p=\rho \vec e_2 & \textrm{if $(t,x)\in \RR_+\times\RR^2$,}\\
\partial_{t}\rho+v\cdot\nabla \rho-\Delta \rho=0 & \textrm{if $(t,x)\in \RR_+\times\RR^2$,}\\ 
\Div v=0, &\\ 
({v},{\rho})_{| t=0}=(v^{0},\rho^0).  
\end{array} \right. \tag{B$_{0}$}
\end{equation} 
Before giving more details about our main contribution  we shall review some aspects of the   vortex patch problem for the viscous/inviscid incompressible fluid. Recall first the classical  Navier-Stokes equations,
\begin{equation}
\left\{ \begin{array}{ll} 
\partial_{t}v+v\cdot\nabla v-\mu\Delta v+\nabla p=0 & \textrm{if $(t,x)\in \RR_+\times\RR^2$,}\\
\Div v=0, &\\ 
{v}_{| t=0}=v^{0}.\tag{NS$_\mu$}  
\end{array} \right.
\end{equation}   
Notice  that the incompressible Euler system (E), denoted sometimes by  (NS$_0$), is given by
\begin{equation}
\left\{ \begin{array}{ll} 
\partial_{t}v+v\cdot\nabla v+\nabla p=0 & \textrm{if $(t,x)\in \RR_+\times\RR^2$,}\\
\Div v=0, &\\ 
{v}_{| t=0}=v^{0}.\tag{E}  
\end{array} \right.
\end{equation} 
We point out that the global existence of classical solutions  for Euler system is based on  the  structure of the vorticity which is transported by the flow, that is, 
\begin{equation*} 
\partial_{t}\omega+v\cdot\nabla\omega=0.
\end{equation*}   
This provides infinite family of conservation laws and in particular we get for all $p\in[1,\infty]$
\begin{equation}\label{conservation}
\|\omega(t)\|_{L^p}=\|\omega^0\|_{L^p}.
\end{equation}
We mention that  the conservation laws \eqref{conservation} served as a suitable framework for Yudovich  \cite{Yd}  to relax the classical hyperbolic theory and show that (NS$_{\mu}$) and (E) are globally well-posed whenever $\omega^0\in L^1\cap L^\infty$. In this pattern, the velocity is no longer in the Lipschitz class but belongs to the $\log-$Lipschitz space, denoted by $LL$\footnote{The space $LL$ is the set of bounded functions $u$ such that $$\| u\|_{LL}\triangleq\sup_{0<\vert x-y\vert<1}\frac{\vert u(x)-u(y)\vert}{\vert x-y\vert\log\frac{e}{\vert x-y\vert}}.$$}. It is known that with this regularity the associated flow $\Psi$ is continuous with respect to $(t,x)-$variables and the vorticity can be recovered from its initial value according to the formula,
\begin{equation}
\omega(t,\Psi(t,x))=\omega^0(x).
\end{equation}
In particular, when the initial vorticity $\omega^0={\bf 1}_{\Omega_0}$ is a vortex patch with $\Omega_0$ being  a regular bounded domain, then the advected vorticity remains a vortex patch relative to a domain $\Omega_t\triangleq\Psi(t,\Omega_0)$ which is homeomorphic to $\Omega_0$.  It is important to emphasize that the regularity persistence of the boundary does not follow from the general theory of Yudovich because the flow is not in general better than $C^{e^{-\alpha t}}$ where $\alpha$ depends on $\omega^0$.  This problem was solved by  Chemin who proved in  \cite{che1} that when the initial boundary is  $C^{1+\EE}$ then the boundary of the  patch keeps this regularity through the time. Broadly speaking, Chemin's strategy  is entirely based on the control of Lipschitz norm of the velocity by means of logarithmic estimate of $\Vert\omega\Vert_{C^{\EE}(X)}$ with $C^{\EE}(X)$ is an  anisotropic H\"older space associated to an adequate family of vector fields that capture the conormal regularity of the velocity(see section \ref{vortex-patch}).

\quad\quad The study for the viscous case was initiated by Danchin in \cite{danchinpoche} who proved that if $\omega^0={\bf 1}_{\Omega_0}$, such that the domain  $\Omega_0$ is  $C^{1+\EE}$ then the velocity $v_{\mu}$ is Lipschitz uniformly with respect to the viscosity $\mu$.  He also showed that the transported vorticity by the viscous flow $\Psi_{\mu}$ remains in the class $C^{1+\varepsilon^\prime}, \forall \varepsilon^\prime<\varepsilon.$  Note that contrary to the H\"olderian regularity, there is no  loss of regularity in the  Besov spaces  $B_{p,\infty}^{\EE}, \; \forall \, p<\infty$. For the borderline case $p=\infty$  Hmidi showed in \cite{hmidipoche} that this loss of regularity is artificial and his proof is mainly related  to  some smoothing effects for the  transport-diffusion equation using Lagrangian coordinates.  There is  a large literature dealing with this subject and some connected topics and for more details we refer the reader to the papers \cite{BC,Dw,Fa, GP,hmidipoche} and the references therein.

\quad\quad It could be interesting to extend some of the foregoing results to the stratified Navier-Stokes system (B$_{\mu}$). The investigation of this system with initial vorticity of  patch type  has been started recently in \cite{HZ-poche} for $\mu=0$. It was proved that   if the boundary of the  initial patch is smooth enough   then the velocity is Lipschitz for any positive time and the transported domain $\Omega_t$ preserves its initial regularity. In addition, the vorticity can be decomposed into a singular part which is a vortex patch term and a regular part, which is deeply related to the smoothing effect for density, i.e. $\omega(t)={\bf 1}_{\Omega_t}+\widetilde{\rho}(t)$. Later, the second author studied in \cite{Z-poche} the same system but the usual dissipation operator  $-\Delta$ is replaced by the critical fractional Laplacian  $(-\Delta)^{\frac12}$. He  obtained  sharper results compared to the incompressible Euler equations \cite{che1,HZ-poche} and describe the asymptotic behavior of the solutions for large time.\\
We are now ready to state the first main result, dealing with the global well-posedness for the system (B$_{\mu}$) under a vortex patch initial data. More precisely, we have:
\begin{Theo}\label{pochetheo} 
Let $\Omega_0$ be a simply connected bounded domain such that its  boundary $\partial\Omega_0$ \mbox{is  $C^{1+\EE}$} with $0<\EE<1$. Let  $\omega_{\mu}^{0}={\bf 1}_{\Omega_0}$ and  $\rho^0_{\mu}\in L^1\cap L^{\infty}$ then the following assertions hold.
\begin{enumerate}
\item[\textnormal{(i)}] The system (B$_{\mu}$) admits a unique global solution $(v_{\mu}, \rho_{\mu})$ such that
$$
(v_{\mu}, \rho_{\mu})\in L^{\infty}_{loc}(\RR_+; \Lip)\times L^{\infty}_{loc}(\RR_+; L^1\cap L^\infty).
$$  

\end{enumerate}
More precisely, there exists $C_0\triangleq C(\EE, \Omega_0)>0$ such that, for all $\mu\in]0,1[$ and for all $t\in\RR_+$ we have
\begin{equation}\label{H}
\Vert\nabla v_{\mu}(t)\Vert_{L^\infty}\le C_0e^{C_0 t\log^2(1+t)}.
\end{equation}
\begin{enumerate}
\item[\textnormal{(ii)}] The boundary of the transported domain $\Omega_{\mu}(t)\triangleq\Psi_{\mu}(t, \Omega_0)$ is  $C^{1+\EE}$ for every $t\ge 0$ uniformly on $\mu$, where $\Psi_{\mu}$ denotes the viscous flow associated to $v_{\mu}$. 
\end{enumerate}
\end{Theo}
Let us give a bunch of comments about Theorem \ref{pochetheo} in the following few remarks.
\begin{rema*} Compared to the incompressible Navier-Stokes system, we see that a Lipschitz norm of the velocity has a logarithmic growth for large time. This is due to the logarithmic factor in the growth of the vorticity, namely we have:
\begin{equation*}
\|\omega_\mu(t)\|_{L^\infty}\le C_0\log^2(1+t).
\end{equation*}
\end{rema*}

\begin{rema*}
When the viscosity $\mu$ is identically zero, we obtain the same result as in \cite{HZ-poche} for the stratified Euler system (B$_0$), that is to say:
\begin{equation}\label{M}
\Vert\nabla v(t)\Vert_{L^\infty}\le C_0 e^{C_0 t\log^2(1+t)}. 
\end{equation}  
\end{rema*}
Now we shall briefly outline the ideas of the proof which is done in the spirit of the pioneering work of Chemin \cite{che1}. In order to get a bound for the quantity $\Vert\nabla v_{\mu}(t)\Vert_{L^\infty}$ we first show that the co-normal regularity of the vorticity $\partial_X\omega_{\mu}$ is controlled in $C^{\EE-1}$, with $0<\EE<1$. We then take advantage of the logarithmic estimate to derive the Lipschitz norm of the velocity, with $X$ is a family of selected  vector fields which satisfies the transport equation,
\begin{equation*}
\partial_{t}X+v_{\mu}\cdot\nabla X=X\cdot\nabla v_{\mu}.
\end{equation*}
As it was pointed in \cite{danchinpoche, hmidipoche} the situation in the viscous case is more delicate than the inviscid one due to  the Laplacian operator does not commute with the family $X$. Actually,   the evolution of the directional derivative $\partial_{X}\omega_\mu$ is governed by an inhomogeneous transport-diffusion equation,
 \begin{equation}\label{eq2v}
  (\partial_{t}+v_{\mu}\cdot\nabla-\mu\Delta)\partial_X\omega_\mu=-\mu[\Delta, X]\omega_\mu+\partial_X\partial_1\rho_{\mu},
   \end{equation}
where $[\Delta, X]$ denotes the commutator between $\Delta$ and $X$. Thus the difficulties reduce to understanding   the  terms $[\Delta, X]\omega_\mu$ and $\partial_X\partial_1\rho_{\mu}$ which apparently need more regularity to be well-defined than what is initially prescribed. To circumvent the problem for the first term we shall use the formalism developed in \cite{danchinpoche,hmidipoche} for $2d-$incompressible Navier-Stokes system. However  to deal with the second term we find more convenient to diagonalize the system written in the vorticity-density formulation and   introduce  the coupled function $\Gamma_{\mu}\triangleq(1-\mu)\omega_{\mu}-\partial_1\Delta^{-1}\rho_{\mu}$ in the spirit of \cite{HR}. This function satisfies the following transport-diffusion equation,
$$
\partial_t\Gamma_{\mu}+v_{\mu}\cdot\nabla\Gamma_{\mu}-\mu\Delta\Gamma_{\mu}=[\partial_1\Delta^{-1}, v_{\mu}\cdot\nabla]\rho_{\mu}\triangleq H_{\mu}.
$$
By applying the directional derivative $\partial_X$ to the last equation we find
\begin{equation*}
(\partial_t +v_{\mu}\cdot\nabla-\mu\Delta) \partial_{X}\Gamma_{\mu}=-\mu[\Delta, X]\Gamma_{\mu}+\partial_X H_{\mu}.
\end{equation*}
At a formal level, and this will be justified rigorously as we shall see in the proofs, we see that $H_{\mu}$ is of order zero with respect to $\rho_{\mu}$ according to the smoothing effect of the singular operator $\partial_1\Delta^{-1}.$ Thus instead of manipulating $\partial_X\partial_1\rho_{\mu}$ in the equation \eqref{eq2v} which consumes two derivatives  we need just to understand $\partial_X H_\mu$  which  exhibits a good behavior on $\rho_{\mu}$ as it was revealed in \cite{HZ-poche}.\\

\quad\quad The second part of this paper is devoted to  the inviscid limit problem which is in fact well-explored for the classical Navier-Stokes system (NS$_{\mu}$). We mention that for smooth initial data the convergence towards Euler equations holds true and the  rate of convergence  in the energy space $L^2$ is bounded by $\mu t$, see \cite{BM} for initial data $v_0\in H^s$ with $s>4$.  In \cite{che0}, Chemin proved a strong convergence in $L^2$ for Yudovich's initial data and obtained that the rate is controlled by $(\mu t)^{\frac12 e^{-Ct}}$, which degenerating in time. To obtain a better result, Constantin and Wu \cite{CW} had to work under vortex patch structure and they obtained $(\mu t)^{\frac12}$. Afterwards, Abidi and Danchin \cite{ad}  improved this result and showed that the rate of convergence is exactly $(\mu t)^{\frac34}$ which is proved to be  optimal for the Rankine vortex. \\
\quad\quad Our second main result reads as follows.
 \begin{Theo}\label{rate} Let $(v_{\mu}, \rho_{\mu})$, $(v, \rho)$, $(\omega_{\mu}, \rho_{\mu})$ and  $(\omega,\rho)$ be the solutions of (B$_{\mu}$), (B$_{0}$), (VD$_{\mu}$) and (VD$_{0}$) respectively with the same initial data such that $\omega_{\mu}^0=\omega^0={\bf 1}_{\Omega_0}$, where $\Omega_0$ is a $C^{1+\EE}$ simply connected bounded  domain. Then for all $t\ge0, \mu\in]0,1[$ and $p\in[2,+\infty[$ we have:  
\begin{enumerate}
\item[(i)] 
 $\Vert v_{\mu}(t)-v(t)\Vert_{L^{p}}+\Vert \rho_{\mu}(t)-\rho(t)\Vert_{L^{p}}\le C_0 e^{e^{C_0 t\log^{2}(2+t)}}(\mu t)^{\frac12+\frac{1}{2p}}.$
\item[(ii)] $\Vert\omega_{\mu}(t)-\omega(t)\Vert_{L^p}\le C_0 e^{e^{C_0 t\log^2(1+t)}}(\mu t)^{\frac{1}{2p}}.$
\end{enumerate}
\end{Theo}

\begin{rema*} When $\rho_{\mu}^0$ and  $\rho^0$ are constants and $p=2$ we get the result of Abidi and \mbox{Danchin \cite{ad}.}
\end{rema*}
The proof of Theorem \ref{rate} will be done using the approach of  \cite{ad} by combining  some classical ingredients like $L^p-$estimates, real interpolation results and some smoothing effects for the density and the vorticity.\\

The last result is dedicated to prove that $(\mu t)^{\frac{1}{2p}}$ is optimal for vortices in the case of Rankine initial data.

\begin{Theo}\label{optimality0}
We assume that $\rho_{\mu}^0$ and $\rho^0$ being constants and $\omega_{\mu}^0=\omega^0={\bf 1}_{\mathbb{D}}$ with $\mathbb{D}$ the unit disc. Then there exists two positive constants $C_1$ and $C_2$ independent on $\mu$ and $t$, such that for $\mu t\le1$, and $p\in[2,+\infty[$ we have:
\begin{equation*}
C_1 (\mu t)^{\frac{1}{2p}}\le\Vert\omega_{\mu}(t)-\omega(t)\Vert_{L^p}\le C_2 (\mu t)^{\frac{1}{2p}}.
\end{equation*}
\end{Theo}

Note that the approach that we shall propose here is different from \cite{ad} which is specific  for $p=2$. The proof of Abidi and Danchin uses the explicit form of  Fourier transform of the Rankine vortex  given through Bessel function combined with its asymptotic behavior. Nevertheless these tools are useless for $p\neq2$ and the alternative is to make the computations in the physical variable using the explicit structure of the heat kernel.

\quad\quad For the reader's convenience, we provide a brief outline of this article. Section 2, starts with  few important results about the Littlewood-Paley decomposition, para-differential calculus and some functional spaces. Moreover, we state some useful technical lemmas, in particular two smoothing effects estimates for  transport-diffusion equations governing  respectively the density and the vorticity evolution. Section 3, mainly treats the general version of Theorem \ref{pochetheo}. Section 4 is divided into two parts. The first one is dedicated to the upper bound rate of convergence. The second part deals with the optimality of the rate of convergence between the vortices. We end this paper with an appendix where we give the proof of some  technical propositions. 

\section{Tools}
Before proceeding, we specify some of the notations we will constantly use during this work. We denote by $C$ a positive constant which may be different in each occurrence but it does not depend on the initial data. We shall sometimes alternatively use the notation $X\lesssim Y$ for an inequality of type $X\le CY$ with $C$ independent of $X$ and $Y$. The notation $C_0$ means a constant depend on the involved norms of the initial data.

\subsection{Littlewood-Paley theory} Our results mostly rely on Fourier analysis methods based on a nonhomogeneous dyadic partition of unity with respect to the Fourier variable. The so-called Littlewood-Paley decomposition enjoying  particularly "nice" properties. These properties are the basis for introducing the important scales of Besov and H\"older spaces and for their study.\\

\quad\quad Let $\chi\in\mathscr{D}(\RR^2)$ be a reference cut-off function, monotonically decaying along rays and so that
\begin{equation*}
\left\{\begin{array}{ll}
\chi\equiv1 & \textrm{if $\|\xi\|\le\frac12$}\\
0\le \chi\le1 & \textrm{if $\frac12\le\|\xi\|\le1$}\\
\chi\equiv0 & \textrm{if $\|\xi\|\ge1.$} 
\end{array}
\right.
\end{equation*}
Define $\varphi(\xi)\triangleq\chi(\frac{\xi}{2})-\chi(\xi)$. We obviously check that $\varphi\ge0$ and $$\supp\varphi\subset\mathcal{C}\triangleq\{\xi\in\RR^2:\frac12\le\|\xi\|\le1\}.$$  
Then we have the following elementary properties, see for example \cite{HCD, che1}

\begin{prop} Let $\chi$ and $\varphi$ be as above. Then the following assertions are hold.
\begin{enumerate}
\item[(1)] Decompositon of the unity: 
$$
\forall\xi\in\RR^2,\quad \chi(\xi)+\sum_{q\ge0}\varphi(2^{-q}\xi)=1.
$$
\item[(2)] Almost orthogonality in the sense of $L^2$:
$$
\forall\xi\in\RR^2,\quad \frac{1}{2}\le\chi^2(\xi)+\sum_{q\ge0}\varphi^2(2^{-q}\xi)\le1.
$$
\end{enumerate}

\end{prop}

The Littlewood-Paley or cut-off operators are defined as follows.
\begin{defi} For every $u\in\mathcal{S}'(\RR^2)$, setting 

\begin{equation*}
\Delta_{-1}u\triangleq\chi(\DD)u,\quad \Delta_{q}u\triangleq\varphi(2^{-q}\DD)u\quad \mbox{if}\;q\in\NN,\quad S_{q}u\triangleq\sum_{j\le q-1}\Delta_{j}u\quad\mbox{for}\; q\ge0.
\end{equation*}
\end{defi}
Some properties of $\Delta_q$ and $S_q$ are listed in the following proposition.
\begin{prop} Let $u,v\in\mathcal{S}'(\RR^2)$ we have
\begin{enumerate}
\item[(i)] $\vert p-q\vert\ge2\Longrightarrow\Delta_p\Delta_q u\equiv0$,
\item[(ii)] $\vert p-q\vert\ge4\Longrightarrow\Delta_q(S_{p-1}u\Delta_p v)\equiv0$,
\item[(iii)] $\Delta_q, S_q: L^p\rightarrow L^p$ uniformly with respect to  $q$ and $p$.
\item[(iv)] 
$$
u=\sum_{q\ge-1}\Delta_q u.
$$
\end{enumerate}  
\end{prop}
Likewise the homogeneous operators $\dot{\Delta}_{q}$ and $\dot{S}_{q}$ are defined by
\begin{equation}\label{Hom}
\forall{q}\in \mathbb{Z}\quad\dot{\Delta}_{q}=\varphi(2^{q}D)u, \quad \dot{S}_{q}=\sum_{ j\le q-1}\dot{\Delta}_{j}v.
\end{equation}
Now, we will recall the definition of the   Besov spaces.
\begin{defi} For $(s,p,r)\in\RR\times[1,  +\infty]^2$. The inhomogeneous Besov space $B_{p,r}^s$ (resp. the homogeneous Besov space $\dot{B}_{p,r}^s$) is the set of all tempered distributions $u\in\mathcal{S}^{'}$ (resp. $u\in\mathcal{S}^{'}_{|{\bf P}})$ such that
\begin{eqnarray*}
&&\Vert u\Vert_{{B}_{p, r}^{s}}\triangleq\Big(2^{qs}\Vert \Delta_{q} u\Vert_{L^{p}}\Big)_{\ell^r}<\infty. \\
&&\big(\mbox{resp. }\Vert u\Vert_{\dot{{B}}_{p, r}^{s}}\triangleq\ \Big(2^{qs}\Vert \dot\Delta_{q} u\Vert_{L^{p}}\Big)_{\ell^r(\mathbb{Z})}<\infty\big). 
\end{eqnarray*}
We have denoted by ${\bf P}$ the set of polynomials.
\end{defi}

\begin{rema*} We notice that:
\begin{enumerate}
\item[(1)] If $s\in\RR_+\backslash\NN$, the H\"older space noted by $C^s$ coincides with $B^s_{\infty,\infty}$.
\item[(2)] $\big(C^s, \|\cdot\|_{C^s}\big)$ is a Banach space coincides with the usual   H\"{o}lder space $C^s$ with equivalent norms, 
\begin{equation}\label{N-equivalent}
\Vert u\Vert_{C^s}\lesssim \Vert u\Vert_{L^\infty}+\sup_{x\neq y}\frac{\vert u(x)-u(y)\vert}{\vert x-y\vert^s}\lesssim\Vert u\Vert_{C^s}.
\end{equation}  
\item[(3)] If $s\in\NN$, the obtained space is so-called H\"older-Zygmund space and still noted by $B^s_{\infty,\infty}$.
\end{enumerate} 
 \end{rema*}

\subsection{Paradifferential calculus}
The well-known  {\it Bony's} decomposition \cite{b} enables us to split formally the product of two tempered distributions $u$ and $v$ into three pieces. In
what follows, we shall adopt the following definition for paraproduct and remainder:
\begin{defi} For a given $u, v\in\mathcal{S}'$ we have
 $$
uv=T_u v+T_v u+\mathscr{R}(u,v),
$$
with
$$T_u v=\sum_{q}S_{q-1}u\Delta_q v,\quad  \mathscr{R}(u,v)=\sum_{q}\Delta_qu\widetilde\Delta_{q}v  \quad\hbox{and}\quad \widetilde\Delta_{q}=\Delta_{q-1}+\Delta_{q}+\Delta_{q+1}.
$$
\end{defi}

The mixed space-time spaces are stated as follows. 
\begin{defi} Let $T>0$ and $(\beta,p,r,s)\in[1, \infty]^3\times\RR$.  We define the spaces $L^{\beta}_{T}B_{p,r}^s$ and $\widetilde L^{\beta}_{T}B_{p,r}^s$ respectively by: 
$$
L^\beta_{T}B_{p,r}^s\triangleq\Big\{u: [0,T]\to\mathcal{S}^{'}; \Vert u\Vert_{L_{T}^{\beta}B_{p, r}^{s}}=\big\Vert\big(2^{qs}\Vert \Delta_{q}u\Vert_{L^{p}}\big)_{\ell^{r}}\big\Vert_{L_{T}^{\beta}}<\infty\Big\},
$$
$$
\widetilde L^{\beta}_{T}B_{p,r}^s\triangleq\Big\{u:[0,T]\to\mathcal{S}^{'}; \Vert u\Vert_{\widetilde L_{T}^{\beta}{B}_{p, r}^{s}}=\big(2^{qs}\Vert \Delta_{q}u\Vert_{L_{T}^{\beta}L^{p}}\big)_{\ell^{r}}<\infty\Big\}.
$$
The relationship between these spaces is given by the following embeddings. Let $ \varepsilon>0,$ then 
\begin{equation}\label{embeddings}
\left\{\begin{array}{ll}
L^\beta_{T}B_{p,r}^s\hookrightarrow\widetilde L^\beta_{T}B_{p,r}^s\hookrightarrow L^\beta_{T}B_{p,r}^{s-\varepsilon} & \textrm{if  $r\geq \beta$},\\
L^\beta_{T}B_{p,r}^{s+\varepsilon}\hookrightarrow\widetilde L^\beta_{T}B_{p,r}^s\hookrightarrow L^\beta_{T}B_{p,r}^s & \textrm{if $\beta\geq r$}.
\end{array}
\right.
\end{equation}
\end{defi}
Accordingly, we have the following interpolation result. 
\begin{cor}
Let $T>0,\; s_1<s<s_2$ and $\zeta\in(0, 1)$ such that $s=\zeta s_1+(1-\zeta)s_2$. Then we have
\begin{equation}\label{m1}
\Vert u\Vert_{\widetilde L_{T}^{a}{B}_{p, r}^{s}}\le C\Vert u\Vert_{\widetilde L_{T}^{a}{B}_{p, \infty}^{s_1}}^{\zeta}\Vert u\Vert_{\widetilde L_{T}^{a}{B}_{p, \infty}^{s_2}}^{1-\zeta}.
\end{equation}
\end{cor}
The following {\it Bernstein} inequalities describe a bound on the derivatives of a function in the $L^b-$norm in terms of the value of the function in the $L^a-$norm, under the assumption that the Fourier transform of the function is compactly supported. For more details we refer \cite{HCD,che1}.
\begin{lem}\label{Bernstein} There exists a constant $C>0$ such that for $1\le a\le b\le\infty$, for every function $u$ and every $q\in\NN\cup\{-1\}$, we have
\begin{enumerate}
\item[(i)]
\begin{equation*}
\sup_{\vert\alpha\vert=k}\Vert\partial^{\alpha}S_{q}u\Vert_{L^{b}}\le C^{k}2^{q\big(k+2\big(\frac{1}{a}-\frac{1}{b}\big)\big)}\Vert S_{q}u\Vert_{L^{a}},\\
\end{equation*}
\item[(ii)]
\begin{equation*}
C^{-k}2^{qk}\Vert\Delta_{q}u\Vert_{L^{a}}\le\sup_{\vert\alpha\vert=k}\Vert\partial^{\alpha}\Delta_{q}u\Vert_{L^{a}}\le C^{k}2^{qk}\Vert\Delta_{q}u\Vert_{L^{a}}.
\end{equation*}
\end{enumerate}
\end{lem}

A noteworthy consequence of Bernstein inequality (i) is the following embedding: 
$$
B_{p,r}^{s}\hookrightarrow B^{\widetilde s}_{\widetilde p,\widetilde r}\quad \textnormal{whenever}\; \widetilde{p}\ge p,
$$
with
$$
\widetilde s<s-2\Big(\frac{1}{p}-\frac{1}{\widetilde p}\Big)\quad\textnormal{or}\quad \widetilde s=s-2\Big(\frac{1}{p}-\frac{1}{\widetilde{p}}\Big) \quad\textnormal{and}\quad \widetilde r\le r.
$$
\subsection{Useful results} This paragraph is reserved to some useful properties freely used throughout this article. The most results concerning the system (VD$_\mu$) are rely strongly on a priori estimates in Besov spaces for the transport-diffusion equation: 
\begin{equation}
\left\{ \begin{array}{ll} 
\partial_{t}a+v\cdot\nabla a-\mu\Delta a=f \\
a_{| t=0}=a^0.  
\end{array} \right.\tag{TD$_\mu$}
\end{equation} 
We start by the persistence of Besov regularity for (TD$_\mu$), which its proof may be found for example in \cite{HCD}.
\begin{prop}\label{prop1} Let $(s, r, p)\in]-1, 1[\times[1, \infty]^2$ and $v$ be a smooth divergence free vector-field.  We assume that $a^0\in{B}_{p, r}^{s}$ and $f\in L_{loc}^{1}(\RR_+; {B}_{p, r}^{s})$. Then for every smooth solution $a$  of (TD$_\mu$) and  $t\geq0$ we have
\begin{equation*}
\Vert a(t)\Vert_{{B}_{p, r}^{s}}\le Ce^{CV(t)}\Big(\Vert a^0\Vert_{{B}_{p, r}^{s}}+\int_{0}^{t}e^{-CV(\tau)}\Vert f(\tau)\Vert_{{B}_{p, r}^{s}}d\tau\Big),
\end{equation*}
with
\begin{equation*}
V(t)=\int_{0}^t\|\nabla v(\tau)\|_{L^\infty}d\tau
\end{equation*}
and $C$  a constant which depends only on $s$ and not on the viscosity.  For the limit case 
\begin{equation*}
s=-1, r=\infty \mbox{ and } p\in[1, \infty]\quad\mbox{ or }\quad s=1, r=1 \mbox{ and } p\in[1, \infty]
\end{equation*}
the above estimate remains true despite we change $V(t)$  by $Z(t)\overset{def}{=}\Vert v\Vert_{L_{t}^{1}{B}_{\infty, 1}^{1}}$. In addition if $a=\textnormal{curl } v$, then the above estimate holds true for all $s\in[1, +\infty[$.
\end{prop}
Next, we state the maximal smoothing effect result for (TD$_\mu$) in mixed time-space spaces, which its proof was developped in \cite{Hmd-Ker}. 

\begin{prop}\label{Persistance} Let $s\in]-1,1[,\;(p_1,p_2,r)\in[1,+\infty]^3$ and $v$ be a divergence free vector field belonging to $L^1_{\loc}(\RR_+;\Lip)$. Then for every smooth solution $a$ of (TD$_\mu$) we have
\begin{equation}\label{maximal-smoothing}
\mu^{\frac{1}{r}}\Vert a\Vert_{\widetilde L^r_t B_{p_1, p_2}^{s+\frac2r}}\le Ce^{CV(t)}(1+\mu t)^{\frac1r}\Big(\Vert a^0\Vert_{B_{p_1, p_2}^{s}}+\Vert f\Vert_{L^1_tB_{p_1, p_2}^{s}}\Big),\quad\forall t\in\RR_+.
\end{equation} 
\end{prop}
The asymptotic behavior in $L^p-$norm with $p\in [2, \infty]$ of every $(\omega_{\mu},\rho_{\mu})$ solution of (VD$_{\mu}$) is given by the following proposition. To be precise we have:
 
\begin{prop}\label{propasyy1}
Let $(\omega_{\mu},\rho_{\mu})$ be a smooth solution of (VD$_{\mu}$) such that $\rho_0\in L^1\cap L^p$ and $ \omega_0\in L^2\cap L^p$ with $p\in[2,\infty].$ Then  for $t\geq 0$,  
$$
\|\omega_\mu(t)\|_{L^p}+\|\nabla\rho_\mu\|_{L^1_tL^p}\le C_0\log^{2-\frac2p}(1+t).
$$
\end{prop}
\begin{rema*} This property has been recently accomplished by \cite{HZ-poche} for Stratified Euler equations (B$_{\mu}$), with $\mu=0$. We point out that the proof of such estimate remains available in our case with minor modifications due to the laplacien term, which has the good sign.
\end{rema*}
We end this paragraph by the Calder\'on-Zygmund estimate which constitute a deep statement of harmonic analysis. 
\begin{prop}\label{CZyg} Let $p\in]1, \infty[$ and $v$ be a divergence-free vector field which its vorticity $\omega\in L^p$.  Then $\nabla v\in L^p$ and 
\begin{equation}
\Vert\nabla v\Vert_{L^p}\le c\frac{p^2}{p-1}\Vert\omega\Vert_{L^p},
\end{equation}  
with $c$ being a universal constant.
\end{prop}

\section{Smooth vortex patch problem} In this section we will give a detailed proof for the first main result stated in Theorem \ref{pochetheo}. We will inspire the  general ideas from Chemin's result, we then follow the argument performed more recently by \cite{HZ-poche, Z-poche} for Stratified Euler system. For this aim, we will state the general framework study of the vortex patch problem.

\subsection{Vortex patch tool box}\label{vortex-patch}
Before entering into details of the proof of the Theorem \ref{pochetheo}, we will state a few important ingredients concerning the study of vortex patch problem. We will start by the concept of an admissible family of vector fields and some relates properties, from which we will derive the notion of anisotropic H\"older space. Afterwards, we state {\it the logarithmic estimate} which is a fundamental tool to prove the Lipschitz property of the velocity.   
\begin{defi}\label{definition-1}
Let $\EE\in]0,1[$. A family of vector fields $X=(X_{\lambda})_{\lambda\in\Lambda}$ is said to be admissible if and only if the following assertions are hold.\\
\begin{itemize}
\item Regularity:
$$\forall\lambda\in \Lambda\quad X_{\lambda}, \Div X_{\lambda}\in C^{\epsilon}.$$
\item Non-degeneray:
\begin{equation}\label{admissibility}
I(X)\triangleq\inf_{x\in\RR^d}\sup_{\lambda\in\Lambda}\big|X_{\lambda}(x)\big|>0.
\end{equation}
\end{itemize}

Setting
\begin{equation}\label{norm}
\widetilde{\|}X_{\lambda}\|_{C^{\EE}}\triangleq\|X_{\lambda}\|_{C^{\EE}}+\|\textnormal{div}X_{\lambda}\|_{C^{\EE}}.
\end{equation}
\end{defi}
\begin{defi} Let $X=(X_{\lambda})_{\lambda\in\Lambda}$ be an admissible family. The action of each factor $X_\lambda$  on $u\in L^\infty$ is defined as the directional derivative of $u$ along $X_\lambda$ by the formula,
\begin{equation*}
\partial_{X_{\lambda}}u=\textnormal{div}(u X_{\lambda})-u\textnormal{div}X_{\lambda}.
\end{equation*}

\end{defi}
The concept of anisotropic H\"{o}lder space, will be noting by $C^\EE(X)$ is defined below. 
\begin{defi}\label{definition}
Let $\EE\in]0,1[$ and  $X$ be an admissible family of vector fields. We say that $u\in C^{\EE}(X)$ if and only if:\\
\begin{itemize}
\item $u\in L^\infty$ and satisfies
\begin{equation*}
\forall\,\lambda\in\Lambda,\; \partial_{X_{\lambda}}u\in
C^{\EE-1}, \quad\sup_{\lambda\in \Lambda}\|\partial_{X_{\lambda}}u\|_{C^{\EE-1}}<+\infty.
\end{equation*}
\item $C^\EE(X)$ is a normed space with
\begin{equation*}
\|u\|_{C^{\EE}(X)}\triangleq\frac{1}{I(X)}\Big(\|u\|_{L^{\infty}}\sup_{\lambda\in\Lambda}\widetilde{\|}X_{\lambda}\|_{C^{\EE}}+\sup_{\lambda\in\Lambda}\|\partial_{X_{\lambda}}u\|_{C^{\EE-1}}\Big).
\end{equation*}
\end{itemize}
\end{defi}

Now, let us take an initial family of vector-field $X_0=(X_{0,\lambda})_{\lambda\in\Lambda}$ and define its time evolution $X_t=\big(X_{t,\lambda})_{\lambda\in\Lambda}$ by
\begin{equation}\label{X-evolution}
X_{t,\lambda}(x)\triangleq X_{0,\lambda}\Psi(t, \Psi^{-1}(t,x)),
\end{equation}
that is $X_t$ is the vector-field $X_0$ transported by the flow $\Psi$ associated to $v$. From this definition the evolution family $X_t$ satisfies the following transport equation.
\begin{prop}\label{pro123}
Let  $v$  be a Lipschitzian vector-field, $\Psi$ its flow and $X_t=(X_{t,\lambda})_{\lambda\in\Lambda}$ is the family defined by \eqref{X-evolution}. Then the following equation holds true.
\begin{equation}\label{hoo}
\left\{\begin{array}{ll}
(\partial_t+v\cdot\nabla)X_{t,\lambda}=\partial_{X_{t,\lambda}}v & \textrm{if $(t,x)\in\RR_+\times\RR^2$}\\
X_{t,\lambda| t=0}=X_{0,\lambda}.
\end{array}
\right.
\end{equation}
\end{prop}
In order to prove the Theorem \ref{pochetheo}, we state the following stationnary logarithmic estimate initially introduced by Chemin \cite{che1}. More precisely,
\begin{Theo}\label{ttt}
Let $\EE\in]0,\,1[$ and $X=(X_\lambda)_{\lambda\in\Lambda}$ be a family of vector fields as in Definition \ref{definition-1}. Let $v$ be a divergence-free vector field such that its vorticity $\omega$ belongs to 
$L^2\cap C^{\EE}(X).$ Then there exists a constant $C$ depending only on $\EE,$ such that 
\begin{equation}\label{5554}
\|\nabla v\|_{L^{\infty}}\leq C\Bigg(\|\omega\|_{L^2}+{\|\omega\|_{L^{\infty}}}\textnormal{log}\bigg(e+\frac{\|\omega\|_{C^{\EE}(X)}}{\|\omega\|_{L^{\infty}}}\bigg)\Bigg).
\end{equation}
\end{Theo} 
We shall now make precise to the boundary regularity and the tangent space used in the proof of Theorem \ref{pochetheo}.
\begin{defi}\label{Char} Let $\EE>0$.
\begin{enumerate}
\item A closed hypersurface $\Sigma$ is said to be $C^{1+\EE}-$regular if there exists $f\in C^{1+\EE}(\RR^2)$ such that $\Sigma$ is a locally zero sets of $f$, i.e., there exists a neighborhood $V$ of $\Sigma$ such that
\begin{equation}\label{Reg-Om}
\Sigma=f^{-1}\{0\}\cap V,\quad \nabla f(x)\ne 0\quad\forall x\in V.
\end{equation}   
\item A vector field $X$ with $C^\EE-$regularity is said to be tangent to $\Sigma$ if $X\cdot\nabla f_{|\Sigma}=0$. The set of such vector field will be noted by $\mathcal{T}_{\Sigma}^{\EE}$.
\end{enumerate}
\end{defi}  
Given a compact hypersurface $\Sigma$ of the class $C^{1+\EE},\;0<\EE<1$. The co-normal space  $C^{\EE}_{\Sigma}$ associated  to $\Sigma$ is defined by
$$
C^{\EE}_{\Sigma}\triangleq\{u\in L^\infty(\RR^2); \forall X\in\mathcal{T}^{\EE}_{\Sigma},\;\Div(Xu)\in C^{\EE-1}\}.
$$
The following Danchin's result stated in \cite{D-Per}, showing that $C^{\EE}_{\Sigma}$ contains the characteristic function of a bounded open domain surruonded by the hypersurface $\Sigma$. More generally we have: 
\begin{prop} Let $\Omega_0$ be $C^{1+\EE}-$bounded domain, with $0<\EE<1$. Then for every function $f\in C^\EE$, we have 
\begin{equation*}
f{\bf 1}_{\Omega_0}\in C^{\EE}_{\Sigma}.
\end{equation*}

\end{prop}
According to the previous Proposition, we strive to give a general version of  the Theorem \ref{pochetheo} which allows to deal with more general structures than the vortex patches. Thus we have: 
\begin{Theo}\label{theo7}
Let $0<\EE<1, X_0$ be a family of admissible vector fields and $v_{\mu}^0$ be a free-divergence vector field such that $\omega^0_{\mu}\in L^2\cap C^{\EE}(X_0).$ Let $\rho^0_{\mu}\in L^1\cap L^\infty$, then for $\mu\in]0,1[$ the system (B$_{\mu}$) admits a unique global solution $(v_\mu,\rho_\mu)\in L^\infty_{loc}(\RR_+; Lip)\times L^\infty(\RR_+; L^{1}\cap L^\infty)$. More precisely:
\begin{equation}
\|\nabla v_{\mu}(t)\|_{L^\infty}\le C_0 e^{C_0 t\log^2(1+t)}.
\end{equation}
Furthermore, we have:
$$
\|\omega_{\mu}(t)\|_{C^\EE(X_t)}+\|\psi_{\mu}(t)\|_{C^\EE(X_t)}\le C_0e^{\exp\{C_0t\log^2(2+t)\}}.
$$

\end{Theo}

\begin{proof} The most difficult point in such proof is to estimate suitably the quantity $\omega_\mu$ in $C^{\EE}(X_t)$ norm. For this aim, we shall use the following coupled function $\Gamma_\mu$ defined by $\Gamma_\mu=(1-\mu)\omega_{\mu}-\mathcal{L}\rho_\mu$, with $\mathcal{L}=\partial_1\Delta^{-1}$. After few computations, we obtain that $\Gamma_{\mu}$ evolves the following inhomogenous transport-diffusion equation: 

\begin{equation*}\label{Gamma0}
(\partial_t+v_{\mu}\cdot\nabla-\mu\Delta)\Gamma_{\mu}=[\mathcal{L}, v_{\mu}\cdot\nabla]\rho_{\mu}.
\end{equation*}
To simplify the notations in what follows, we temporarily suppress the viscosity parameter $\mu$. \\In vertu of \eqref{hoo} of the Proposition \ref{pro123}, one can check that the quantity $\partial_{X_{t,\lambda}}\Gamma$ satisfies the equation,
\begin{equation}\label{Youyou0}
(\partial_t+v\cdot\nabla-\mu\Delta)\partial_{X_{t,\lambda}}\Gamma=X_{t,\lambda}\{[\mathcal{L},v\cdot\nabla ]\rho\}-\mu[\Delta, X_{t,\lambda}]\Gamma.
\end{equation}
According to \cite{danchinpoche,hmidipoche}, the commutator $[\Delta, X_{t,\lambda}]$ can be decomposed as the sum of two terms in the following way:
\begin{equation*}
\mu[\Delta, X_{t,\lambda}]\Gamma=F+\mu G,
\end{equation*}

with

\begin{equation*}
F\triangleq 2\mu T_{\nabla X^i_{t,\lambda}}\partial_i\nabla\Gamma+2\mu T_{\partial_i\nabla\Gamma}\nabla X^i_{t,\lambda}+\mu T_{\Delta X^i_{t,\lambda}}\partial_i\Gamma+ \mu T_{\partial_i\Gamma}\Delta X^i_{t,\lambda}.
\end{equation*}
and

\begin{equation*}
G\triangleq2\mathscr{R}(\nabla X^i_{t,\lambda},\partial_i\nabla\Gamma)+\mathscr{R}(\Delta X^i_{t,\lambda},\partial_i\Gamma).
\end{equation*}
Here, we have used Enstein's convention for the summation over the reapeted indexes. Thus the equation \eqref{Youyou0} takes the following form,

\begin{equation*}\label{Youyou1000}
(\partial_t+v\cdot\nabla-\mu\Delta)\partial_{X_{t,\lambda}}\Gamma=X_{t,\lambda}\{[\mathcal{L},v\cdot\nabla ]\rho\}-(F+\mu G),
\end{equation*}
Applying Theorem 3.38 page 162 in \cite{HCD}, one gets 
\begin{eqnarray}\label{MHOAAA}
\big\|\partial_{X_{\lambda}}\Gamma\big\|_{L^\infty_t C^{\EE-1}}&\le& Ce^{CV(t)}\Big(\|\partial_{X_{0,\lambda}}\Gamma^0\|_{C^{\EE-1}}+\|\partial_{X_{\lambda}}\{\big[\mathcal{L},v\cdot\nabla]\rho\}\|_{ L^1_t C^{\EE-1}}\\
\nonumber&&+(1+\mu t)\|F\|_{L^\infty_t C^{\EE-3}}+\mu\Vert G\Vert_{\widetilde L^1_t C^{\EE-1}}\Big).
\end{eqnarray}
Recall from \cite{HCD, hmidipoche} the following two inequalities
\begin{equation*}
\Vert F\Vert_{L^\infty_t C^{\EE-3}}\le C\Vert\Gamma\Vert_{L^\infty_t L^\infty}\Vert X_{\lambda}\Vert_{L^\infty_t C^{\EE}}.
\end{equation*}
and
\begin{equation*}
\Vert G\Vert_{\widetilde L^1_t C^{\EE-1}}\le C\Vert\Gamma\Vert_{\widetilde L^1_t B^2_{\infty,\infty}}\Vert X_{\lambda}\Vert_{L^\infty_t C^{\EE}}. 
\end{equation*}
Combining with \eqref{MHOAAA}, one finds
\begin{eqnarray}\label{Essential}
\big\|\partial_{X_{\lambda}}\Gamma\big\|_{L^\infty_t C^{\EE-1}}&\le& Ce^{CV(t)}\Big(\|\partial_{X_{0,\lambda}}\Gamma^0\|_{C^{\EE-1}}+\|\partial_{X_{\lambda}}\{\big[\mathcal{L},v\cdot\nabla]\rho\}\|_{L^1_t C^{\EE-1}}\\
\nonumber&&+(1+\mu t)\Vert\Gamma\Vert_{L^\infty_t L^\infty}\Vert X_{\lambda}\Vert_{L^\infty_t C^{\EE}}+\mu\Vert\Gamma\Vert_{\widetilde L^1_t B^2_{\infty,\infty}}\Vert X_{\lambda}\Vert_{L^\infty_t C^{\EE}}\Big).
\end{eqnarray}

{\bf $\bullet$ Estimate of $\|\partial_{X_{0, \lambda}}\Gamma^{0}\big\|_{C^{\varepsilon-1}}$}. From the definition of the function $\Gamma$ we have: 
\begin{equation}\label{YYYY}
\big\|\partial_{X_{0, \lambda}}\Gamma^0\big\|_{C^{\EE-1}}\le\big\|\partial_{X_{0, \lambda}}\omega^0\big\|_{C^{\EE-1}}+\big\|\partial_{X_{0, \lambda}}\mathcal{L}\rho^0\big\|_{C^{\EE-1}}.
\end{equation}
Since $C^\EE$ is an algebra, then we obtain the general result
\begin{eqnarray}\label{algebra}
\|\partial_{X_{\lambda}}u\|_{C^{\varepsilon-1}}&\le&\|\textnormal{div}(uX_{\lambda})\|_{C^{\EE-1}}+\|u\textnormal{ div}X_{\lambda}\|_{C^{\EE-1}}\\
\nonumber&\lesssim&\|uX_{\lambda}\|_{C^{\EE}}+\|u\textnormal{ div}X_{\lambda}\|_{L^\infty}\\
\nonumber&\lesssim& \|u\|_{C^{\varepsilon}}\widetilde{\|}X_\lambda\|_{C^\EE}.
\end{eqnarray}
Consequently
\begin{equation*}
\big\|\partial_{X_{0, \lambda}}\omega^0\big\|_{C^{\EE-1}}\lesssim \widetilde{\|}X_{0, \lambda}\|_{C^\EE}\|\omega^0\|_{C^\EE},\quad\big\|\partial_{X_{0, \lambda}}\mathcal{L}\rho^0\big\|_{C^{\EE-1}}\lesssim\widetilde{\|}X_{0,\lambda}\|_{C^\EE}\|\mathcal{L}\rho^0\|_{C^\EE}.
\end{equation*}
Concerning $\|\mathcal{L}\rho^0\|_{C^\EE}$, using the fact that $\mathcal{L}$ is of order $-1$. Then Bernstein's inequality yields for $p\geq \frac{2}{1-\EE}$,
\begin{eqnarray}\label{Toto}
\|\mathcal{L}\rho^0\|_{C^\EE}&\le&\|\mathcal{L}\rho^0\|_{L^\infty}+\sup_{q\in\NN}2^{q\EE}\|\Delta_q\mathcal{L}\rho^0\|_{L^\infty}\\ 
\nonumber&\lesssim&\|\mathcal{L}\rho^0\|_{L^\infty}+\sup_{q\in\NN}2^{q(\EE-1+2/p)}\|\Delta_q\rho^0\|_{L^p}.
\end{eqnarray}
Furtheremore $\mathcal{L}$ have a non local structure, i.e., 
\begin{equation*}
\mathcal{L}\rho(t,x)\triangleq\int_{\RR^2}\frac{(x_1-y_1)}{\vert x-y\vert^2}\rho(t,y)dy,
\end{equation*}
and so
\begin{equation*}
\vert\mathcal{L}\rho(t,x)\vert\le\int_{\RR^2}\frac{\vert\rho(t,y)\vert}{\vert x-y\vert} dy=\bigg(\frac{1}{\vert \cdot\vert}\star \vert\rho(t,\cdot)\vert\bigg)(x).
\end{equation*}
Applying the convolution product properties and $\Vert\rho(t)\Vert_{L^{1}\cap L^\infty}\le\Vert\rho^0\Vert_{L^{1}\cap L^\infty}$, we obtain
\begin{eqnarray}\label{V3}
\Vert\mathcal{L}\rho(t)\Vert_{L^\infty}&\lesssim&\Vert\rho(t)\Vert_{L^{1}\cap L^\infty}\\
\nonumber&\lesssim& \Vert\rho^0\Vert_{L^{1}\cap L^\infty}.
\end{eqnarray}
Putting together \eqref{Toto} and \eqref{V3}. Then in view of $\Delta_q:L^p\rightarrow L^p$ is continuous and $L^p=[L^1,L^\infty]_{\frac1p}$, we deduce
\begin{equation*}
\|\mathcal{L}\rho^0\|_{C^\EE}\le\|\rho^0\|_{L^1\cap L^\infty}.
\end{equation*}
Therefore
\begin{equation*}
\big\|\partial_{X_{0, \lambda}}\mathcal{L}\rho^0\big\|_{C^{\varepsilon-1}}\le C_0\widetilde{\|}X_{0, \lambda}\|_{C^\EE}
\end{equation*}
together with \eqref{YYYY}, it happens
\begin{equation}\label{Ham}
\big\|\partial_{X_{0, \lambda}}\Gamma^0\big\|_{C^{\EE-1}}\le C_0\widetilde\|X_{0, \lambda}\|_{C^{\EE}}
\end{equation}
{\bf $\bullet$ Estimate of $\|\partial_{X_{\lambda}}\{\big[\mathcal{L},v\cdot\nabla]\rho\}\|_{L^1_t C^{\EE-1}}$}. To estimate this term we write again in view of \eqref{algebra}, 
\begin{equation*}
\|\partial_{X_{\lambda}}\{\big[\mathcal{L},v\cdot\nabla]\rho\}\|_{ L^1_t C^{\EE-1}}\lesssim C\widetilde{\|}X_{t,\lambda}\|_{ L^\infty_t C^\EE}\Big\|\big[\mathcal{L},v\cdot\nabla  \big]\rho\Big\|_{ L^1_t C^\EE}.
\end{equation*}
Then in accordance with the Proposition \ref{An1} stated in appendix, the last estimate becomes
\begin{equation}\label{Amina0}
\|\partial_{X_{t,\lambda}}\{\big[\mathcal{L},v\cdot\nabla]\rho\}\|_{L^1_t C^{\EE-1}}\le C_0\widetilde{\|}X_{t,\lambda}\|_{L^\infty_{t} C^\EE}t.
\end{equation} 
{\bf $\bullet$ Estimate of $\Vert\Gamma\Vert_{L^\infty_t L^\infty}$.} By definition we have for $\mu\in]0,1[$, 
\begin{equation}\label{V11111}
\|\Gamma\|_{L^\infty_t L^\infty}\le \|\omega\|_{L^\infty_t L^\infty}+\|\mathcal{L}\rho\|_{L^\infty_t L^\infty}.
\end{equation}
Thanks to the Proposition \ref{propasyy1} we have,
\begin{equation*}
\Vert\omega(t)\Vert_{L^\infty}\le C_0\log^2(2+t).
\end{equation*}
Note that the term $\|\mathcal{L}\rho\|_{L^\infty_t L^\infty}$ will be done exactly as in \eqref{V3}. Then in view of the last estimate, \eqref{V11111} takes the form
\begin{equation}\label{V1}
\Vert\Gamma\Vert_{L^\infty_t L^\infty}\le C_0\log^2(2+t).
\end{equation}

{$\bullet$ \bf Estimate of $\|\Gamma\|_{\widetilde L^1_t B^{2}_{\infty,\infty}}$}. Applied the maximal smoothing effect \eqref{maximal-smoothing} to the equation \eqref{Gamma0}, it happens 
\begin{equation*}
\mu\|\Gamma\|_{\widetilde L^1_t B^{2}_{\infty,\infty}}\le Ce^{CV(t)}(1+\mu t)\Big(\|\Gamma^0\|_{B^{0}_{\infty,\infty}}+\big\|\big[\mathcal{L},v\cdot\nabla  \big]\rho\big\|_{L^1_t B^0_{\infty,\infty}}\Big).
\end{equation*}
Using the fact $L^\infty\hookrightarrow B^{0}_{\infty,\infty}$ and $C^\EE\hookrightarrow B^{0}_{\infty,\infty}$ for $\EE>0$, it follows
\begin{equation*}
\mu\|\Gamma\|_{\widetilde L^1_t B^{2}_{\infty,\infty}}\le Ce^{CV(t)}(1+\mu t)\Big(\|\Gamma^0\|_{L^{\infty}}+\big\|\big[\mathcal{L},v\cdot\nabla  \big]\rho\big\|_{L^1_t C^\EE}\Big), 
\end{equation*}
and, in turn, using once more the Proposition \ref{An1}, we get
\begin{equation}\label{Youyou-Es}
\mu\|\Gamma\|_{\widetilde L^1_t B^{2}_{\infty,\infty}}\le Ce^{CV(t)}(1+\mu t)\big(\|\Gamma^0\|_{L^{\infty}}+C_0t\big), 
\end{equation}
For $\|\Gamma^0\|_{L^{\infty}}$, applying the same argument as in \eqref{V3}, we deduce
\begin{equation*}
\Vert\Gamma^0\Vert_{L^\infty}\le\Vert\omega^0\Vert_{L^\infty}+\Vert\rho^0\Vert_{L^1\cap L^\infty}.
\end{equation*}
Together with \eqref{Youyou-Es}, it holds that for $\mu\in]0,1[$ 
\begin{equation}\label{V220}
\mu\|\Gamma\|_{\widetilde L^1_t B^{2}_{\infty,\infty}}\le C_0e^{CV(t)}(1+t)^2.
\end{equation}

Plugging \eqref{Ham}, \eqref{Amina0}, \eqref{V1}, \eqref{V220} in \eqref{Essential}, then after few computations we obtain for $\mu\in]0,1[$

\begin{eqnarray}\label{Gamma}
\big\|\partial_{X_{\lambda}}\Gamma\big\|_{L^\infty_t C^{\EE-1}}&\le & C_0e^{CV(t)}
(1+t^2)\log^2(2+t)\big(1+\widetilde{\|} X_{\lambda}\|_{L^{\infty}_t C^{\EE}}\big).
\end{eqnarray}
But, 
\begin{equation*}
\big\| \partial_{X_{t,\lambda}}\omega(t)\big\|_{C^{\varepsilon-1}}\le \big\| \partial_{X_{t,\lambda}}\Gamma(t)\big\|_{C^{\varepsilon-1}}+\big\| \partial_{X_{t,\lambda}}\mathcal{L}\rho(t)\big\|_{C^{\varepsilon-1}}
\end{equation*}
combined with \eqref{algebra}, \eqref{V3} and \eqref{Gamma} we get
\begin{eqnarray}\label{Youyou555}
\big\| \partial_{X_{t,\lambda}}\omega(t)\big\|_{C^{\varepsilon-1}}&\le & C_0e^{CV(t)}(1+t^2)\log^2(2+t)\big(1+\widetilde{\|} X_{\lambda}\|_{L^{\infty}_t C^{\EE}}\big)+C_0\widetilde{\|} X_{\lambda}\|_{L^{\infty}_t C^{\EE}}\\
\nonumber&\le& C_0e^{CV(t)}(1+t^2)\log^2(2+t)\big(1+\widetilde{\|} X_{\lambda}\|_{L^{\infty}_t C^{\EE}}\big).
\end{eqnarray}
The term $ \widetilde{\|}{X_\lambda}\|_{L^{\infty}_t C^\EE}$ may be bounded by taking advantage to \eqref{hoo} and the Proposition \ref{prop1}, we thus have
\begin{equation}\label{Youyou050}
\|X_{t,\lambda}\|_{C^\EE}\leq C e^{CV(t)}\Big(\|X_{0,\lambda}\|_{C^\EE}+\int_{0}^{t}e^{-CV(\tau)}\|X_{\tau,\lambda}v(\tau)\|_{C^\EE}d\tau\Big).
\end{equation}
According to \cite{HCD,che1}, the quantity $\|\partial_{X_{t, \lambda}}v\|_{C^\EE}$ satisfies,
\begin{equation*}
\|\partial_{X_{t, \lambda}}v\|_{C^\EE}\le C\big(\|\partial_{X_{t,\lambda}}\omega\|_{C^{\EE-1}}+\|\Div X_{t,\lambda}\|_{C^\EE}\|\omega(t)\|_{L^{\infty}}+\|X_{t,\lambda}\|_{C^\EE}\|\nabla v(t)\|_{L^{\infty}}\big).
\end{equation*}
Hence, \eqref{Youyou050} and the last estimate combined with the Gronwall inequality yield 
\begin{equation}\label{Youyou052}
\|X_{t,\lambda}\|_{C^\EE}\leq C e^{CV(t)}\Big(\|X_{0,\lambda}\|_{C^\EE}+\int_{0}^{t}e^{-CV(\tau)}\big(\|\partial_{X_{\tau,\lambda}}\omega(\tau)\|_{C^{\EE-1}}+\|\Div X_{\tau,\lambda}\|_{C^\EE}\|\omega(\tau)\|_{L^{\infty}}\big)\Big)d\tau.
\end{equation}
To conclude, it is enough to treat the term $\Div X_{t,\lambda}$. To do this, we apply "$\Div $" to (\ref{hoo}) and using the fact $\Div v=0$, we deduce that $\Div X_{t,\lambda}$ evolves the equation
$$(\partial_t+v\cdot\nabla)\Div X_{t,\lambda}=0.$$
Again the Proposition \ref{prop1} gives
\begin{equation}\label{68}
\|\textnormal{div}X_{t,\lambda}\|_{C^\EE}\le Ce^{CV(t)}\|\Div X_{0,\lambda}\|_{C^\EE}.
\end{equation}
Combining \eqref{Youyou052} and \eqref{68}, we shall get
\begin{eqnarray*}
\widetilde{\|}X_{t,\lambda}\|_{C^\EE}&\le& C e^{CV(t)}\bigg(\|X_{0,\lambda}\|_{C^\EE}+\|\Div X_{\lambda}(0)\|_{C^\EE}\|\omega\|_{L^1_t L^{\infty}}+\int_{0}^{t}e^{-CV(\tau)}\|\partial_{X_{\tau,\lambda}}\omega(\tau)\|_{C^{\EE-1}}d\tau\bigg).
\end{eqnarray*}
Then, the Proposition \ref{propasyy1} implies
\begin{eqnarray*}
\widetilde{\|}X_{t,\lambda}\|_{C^\EE}&\le& C_0 e^{CV(t)}\bigg(\widetilde{\|}X_{0,\lambda}\|_{C^\EE}\log^2(2+t)+
\int_{0}^{t}e^{-CV(\tau)}\|\partial_{X_{\tau,\lambda}}\omega(\tau)\|_{C^{\EE-1}}d\tau\bigg)\\
\nonumber&& \le C_0 e^{CV(t)}\bigg(\log^2(2+t)+
\int_{0}^{t}e^{-CV(\tau)}\|\partial_{X_{\tau,\lambda}}\omega(\tau)\|_{C^{\EE-1}}d\tau\bigg),
\end{eqnarray*}
combined with \eqref{Youyou555}, it holds
\begin{eqnarray*}
e^{-CV(t)}\widetilde{\|}X_{t,\lambda}\|_{C^\EE} &\le& C_0\bigg((1+t^2)(1+t)\log^2(2+t)+
(1+t^2)\log^2(2+t)\int_{0}^{t}e^{-CV(\tau)}\widetilde{\|} X_{\tau,\lambda}\|_{L^{\infty}_{\tau} C^{\EE}}d\tau\bigg) ,
\end{eqnarray*}
By means of Gronwall inequality, it follows by few computations that 
\begin{equation*}
\widetilde{\|}X_{t,\lambda}\|_{C^\EE}\le C_0 e^{C_0 t^3}e^{CV(t)}.
\end{equation*}
Accordingly \eqref{Youyou555} becomes
\begin{equation*}
\|\partial_{X_{t,\lambda}}\omega(t)\|_{C^{\EE-1}}\le C_0 e^{C_0 t^3}e^{CV(t)}.
\end{equation*}

Putting together the last two estimates, we end up with
\begin{equation}\label{desired-estimate}
\widetilde{\|}X_{t,\lambda}\|_{C^\EE}+\|\partial_{X_{t,\lambda}}\omega(t)\|_{C^{\EE-1}}\le C_0 e^{C_0 t^3}e^{CV(t)}.
\end{equation}
According to the definition \ref{definition}, we recall that:
\begin{equation}\label{finally}
\|\omega(t)\|_{C^\EE(X_t)}=\frac{1}{I(X_t)}\Big(\|\omega\|_{L^\infty}\sup_{\lambda\in\Lambda}\widetilde{\|}X_{\lambda}(t)\|_{C^\EE}+\sup_{\lambda\in\Lambda}\|\partial_{X_{t,\lambda}}\omega(t)\|_{C^{\EE-1}}\Big)
\end{equation}
The required estimate for $\omega$ in $C^\EE(X_t)$ norm follows by showing that $X_t$ defined in \eqref{X-evolution} is a non degenerate family, that is to say, $I(X_t)>0$ . For that purpose, we derive $X_{\lambda}\circ\Psi(t,x )\triangleq \partial_{X_{0,\lambda}}\Psi(t,x)$ with respect to time and using the fact 
\begin{equation*}
\left\{\begin{array}{ll}
\frac{\partial}{\partial t}\Psi(t,x)=v(t,\Psi(t,x))\\
\Psi(0,x)=x,
\end{array}
\right.
\end{equation*}  
it follows
\begin{equation}\label{I1}
\left\{\begin{array}{ll}
\partial_t \partial_{X_{0,\lambda}}\Psi(t,x)=\nabla v(t,\Psi(t,x))\partial_{X_{0,\lambda}}\Psi(t,x)\\
X_{0,\lambda}\Psi(0,x)=X_{0,\lambda}.
\end{array}
\right.
\end{equation}  

Or, \eqref{X-evolution} ensures that \eqref{I1} is time reversibile. In other words,
\begin{equation*}
\partial_t X_{t,\lambda}(x)=X_{0,\lambda}(x)+\int_0^t X_{\tau,\lambda}(x)\cdot\nabla v(\tau,x)d\tau.
\end{equation*}
Gronwall's inequality yields

\begin{equation*}
\vert X_{t,\lambda}(x)\vert^{\pm1}\le X_{0,\lambda}(x)e^{V(t)}
\end{equation*}
in accordance with \eqref{admissibility}, we readily get 
\begin{equation}\label{index-non-degenerate}
I(X_t)\ge I(X_0)e^{-V(t)}>0.
\end{equation}
Consequently, in view of \eqref{desired-estimate}, \eqref{finally} and the last estimate, we finally obtain 
\begin{equation}\label{desired-estimate1}
\|\omega(t)\|_{C^\EE(X_t)}\le C_0 e^{C_0 t^3}e^{CV(t)}.
\end{equation}
Now, we are in position to apply the logarithmic estimate \eqref{5554}. By virtue of \eqref{desired-estimate1}, the increasing of the function $\zeta\mapsto\zeta\log(e+a/\zeta)$ and the Proposition \eqref{propasyy1}, it holds
\begin{eqnarray*}
\|\nabla v(t)\|_{L^\infty}&\le&C_0\Big(\log(2+t)+\log^2(2+t)\log\big(e+\|\omega(t)\|_{C^\EE(X_t)}\big)\Big)\\
&\le& C_0\Big(1+t^3\log^2(2+t)+\log^2(2+t)\int_0^t\|\nabla v(\tau)\|_{L^\infty}d\tau\Big).
\end{eqnarray*}
Hence, Gronwall's inequality yields
\begin{equation*}\label{eq2}
\|\nabla v(t)\|_{L^\infty}\le C_0 e^{C_0 t\log^2(2+t)},
\end{equation*}
combining this estimate with \eqref{desired-estimate1}, we get
$$
\|\omega(t)\|_{C^\EE(X_t)}\le  C_0e^{\exp{C_0 t\log^2(2+t)}}.
$$
To finalize, let us estimate $\Psi(t)$ in $C^\EE(X_t)$. First, we employ that $ \partial_{X_{t,\lambda}}\Psi(t)=X_{t,\lambda}\circ\Psi(t)$ for every $\lambda\in\Lambda$, then in vertu of \eqref{N-equivalent} we thus have
\begin{eqnarray*}
\|X_{t,\lambda}\circ\Psi(t)\|_{C^\EE}&\le& \|X_{t,\lambda}\|_{C^\EE}\|\nabla\Psi(t)\|_{L^\infty}^\EE\\
&\le&  \|X_{t,\lambda}\|_{C^\EE} e^{CV(t)}\quad \forall\lambda\in\Lambda.
\end{eqnarray*}
Here we have used the classical estimate $e^{-CV(t)}\le\|\nabla\Psi^{\pm 1}(t)\|_{L^\infty}\le e^{CV(t)}$. Hence 
\begin{equation}\label{globat}
\|\Psi(t)\|_{C^\EE(X_t)}\le C_0e^{\exp{C_0 t\log^2(2+t)}},
\end{equation}
this concludes the proof.
\end{proof}

\subsection{Proof of  Theorem \ref{pochetheo}.} The proof of the Theorem \ref{pochetheo} requires two principal steps:\\
\begin{enumerate}
\item[(1)] The velocity vector fields is a Lipschitz function globally in time, which immediately follows from Theorem \ref{theo7}. 
\item[(2)] The persistence of H\"olderian regularity in time of the transported patch, i.e., $\partial\Omega_t$ is a simple curve with $C^{1+\EE}-$regularity given by the following scheme: 
\begin{itemize}
\item[(2.i)] Fabricate an initial admissible family $X_0=(X_{0,\lambda})_{0\le\lambda\le1}$, which enables us to show that ${\bf{1}}_{\Omega_0}\in C^\EE(X_0)$ and parametrize its boundary $\partial\Omega_0$ by a simple curve.
\item[(2.ii)] The regularity of evolution family $X_t=(X_{t,\lambda})_{0\le\lambda\le1}$ and the boundary $\partial\Omega_t$, with $\Omega_t=\Psi(t,\Omega_0)$.
\end{itemize}
\end{enumerate}
{\bf(2.i)} Since $\partial\Omega_0$ is an hypersurface of the class $C^{1+\EE}$. Consequently, (1) of the definition \ref{Char} ensures the existence of a local chart $(f_0, V_0)$, with $V_0$ is a neighborhood of $\partial\Omega_0$ such that
\begin{equation*}
\left\{\begin{array}{ll}
f_0\in C^{1+\EE}(\RR^2),\quad \nabla f_0(x)\neq0 & \textrm{on $V_0$}\\
\partial \Omega_0=f_0^{-1}(\{0\})\cap V_0,
\end{array}
\right.
\end{equation*} 
On the other hand, let $\chi\in\mathscr{D}(\RR^2), 0\le\chi\le1$ and
\begin{equation*}
\supp\chi\subset V_0,\quad\chi(x)=1 \quad\forall x\in W_{0},
\end{equation*} 
where $W_0$ is  a small neighborhood of $\partial\Omega_0$ such that $W_0\Subset V_0$. Next, define for every $x\in\RR^2$ the family $X_0=(X_{0,\lambda})_{0\le\lambda\le1}$ by:
\begin{equation}\label{initial-fam}
X_{0,0}(x)=\nabla
^{\bot}f_0(x)\quad\textnormal{and}\quad X_{0,1}(x)=(1-\chi(x))\left(\begin{array}{c}
1\\
0\\
\end{array}
\right).
\end{equation}
It is worthwhile to examine the admissibility of the family $X_0=(X_{0,\lambda})_{0\le\lambda\le1}$. First, we obviously check that $X_0=(X_{0,\lambda})_{0\le\lambda\le1}$ is non-degenerate, and that each component $X_{0,\lambda}$ and its divergence are in $C^\EE(\RR^2)$, then according to Definition \ref{definition-1}, we conclude that $X_0=(X_{0,\lambda})_{0\le\lambda\le1}$ is an admissible family.\\
Second, $X_0=(X_{0,\lambda})_{0\le\lambda\le1}$ is a tangential family (see, (2)-Definition \ref{Char}) with respect to $\Sigma=\partial\Omega_0$, i.e., $$X_{0,\lambda}\in\mathcal{T}^\EE_{|\Sigma},\quad\forall\lambda\in\{0,1\}.$$ Indeed, for the component  $X_{0,0}$, clearly we have:
\begin{equation*}
X_{0,0}(x)\cdot\nabla f_0(x)=\nabla^{\perp}f_0(x)\cdot\nabla f_0(x)=0,\quad\forall x\in\partial\Omega_0,
\end{equation*}
while for the component $X_{0,1}$, using the fact $\chi\equiv1$ on $W_0$, we immediately obtain
\begin{eqnarray*}
X_{0,1}(x)\cdot\nabla f_0(x)&=&(1-\chi(x))\partial_1 f_0(x)\\
&=&0.
\end{eqnarray*}

{\bf(2.ii)} For every $\lambda\in\{0,1\}$ and $x\in\RR^2$, we set $X_{t,\lambda}(x)=\partial_{X_{0,\lambda}}\Psi\big(t,\Psi^{-1}(t,x)\big)$. Discussing by the same argument as in \eqref{68}, \eqref{desired-estimate} and \eqref{index-non-degenerate}, we infer that  $(X_t)$ still remains non-degenerate for every $t\ge0$, and that each $X_{t,\lambda}$ still has components and
divergence in $C^{\EE}$. This means that $X_{t}=(X_{t,\lambda})_{0\le\lambda\le1}$ is an admissible family for all $t\ge0$. \\
Now, we will parametrize the boundary $\partial\Omega_0$. To do this, let $x_0\in\partial\Omega_0$ and define the curve $\gamma ^0$ by the following ordinary  differential equation
 $$\left\lbrace
\begin{array}{l}
\partial_{\sigma}\gamma ^0(\sigma)=X_{0,0}(\gamma ^0(\sigma))\\
 \gamma ^0(0)=x_0.\\
\end{array}
\right.$$
By classical arguments we can see that  $\gamma ^0$ belongs to $C^{1+\EE}(\RR,\RR^2)$. A natural way to define the evolution parametrization of $\partial\Omega_t$ is to set for every $t\ge0$, 
\begin{equation*}
\gamma(t,\sigma)\triangleq\Psi(t,\gamma_0(\sigma)).
\end{equation*}
Clearly that $\gamma(t,\cdot)$ is the transported of $\gamma_0$ by the flow $\Psi$. By applying the criterion differentiation with respect to $\sigma$, we readily get 
$$
\partial_{\sigma}\gamma(t,\sigma)=\big(\partial_{X_{0,0}}\Psi\big)(t,\,\gamma ^0(\sigma)).$$

On the other hand, $\partial_{X_{0,0}}\Psi\equiv X_{0,0}\circ\Psi$, thus we have from estimate \ref{globat} of the Theorem \ref{theo7} that $\partial_{X_{0,0}}\psi\in L^{\infty}_{loc}(\RR_+;\,C^{\EE})$, accordingly $\gamma(t)$ belongs to  $L^{\infty}_{loc}(\RR_+;\,C^{1+\EE})$. This tells us the regularity  persistence of the boundary $\partial\Omega_t$ and so the proof of the Theorem \ref{pochetheo} is accomplished.

\section{The rate convergence} 
\subsection{General statement} In this paragraph we are interested in the rate convergence between $(v_\mu,\rho_\mu)$ and $(v,\rho)$, the solutions of (B$_\mu$) and (B$_0$). To be precise, we will provide a more general version of the Theorem \ref{rate}. For this purpose, we state the following auxiliary result which shows that any vortex patch for smooth bounded domain belongs to $\dot B^{\frac{1}{p}}_{p,\infty}$.  

\begin{prop}\label{Inj}
Let $\Omega_0$ be a $C^{1+\EE}-$bounded domain, with $0<\EE<1$, then the function ${\bf 1}_{\Omega_0}$ belongs to $\dot B^{\frac{1}{p}}_{p,\infty}$. 
\end{prop}
\begin{proof}
We follow the formalism performed in \cite{M-rate} with more details. Since $\Omega_0$ is $C^{1+\EE}-$bounded domain, then in view of $C^{1+\EE}\hookrightarrow Lip$ we deduce that ${\bf 1}_{\Omega_0}\in L^\infty\cap BV$, with $BV$\footnote{$BV$ is the space of functions of bounded variations defined by $$ BV(\RR^2)\triangleq\Big\{u\in L^1(\RR^2):\forall i=1,\ldots,2,\;\exists\lambda_i\in\mathscr{M}_b(\RR^2,\RR^2); \int_{\RR^2}u\frac{\partial\varphi}{\partial x_i}dx=-\int_{\RR^2}\varphi d\lambda_i\quad\forall\varphi\in\mathscr{D}(\RR^2)\Big\}$$ equipped with the norm $$\Vert u\Vert_{BV}\triangleq\|u\|_{L^1}+\vert D u\vert(\RR^2),$$ where $\vert D u\vert(\RR^2)$ is the total variation of measure $Du$.} is the Banach space of bounded variations. By means of the Proposition \ref{injec} stated in the appendix, we get

$$BV\hookrightarrow \dot B^1_{1,\infty}$$
In particular we have  
\begin{equation*}
\Vert \dot\Delta_q{\bf 1}_{\Omega_0}\Vert_{L^1}\lesssim 2^{-q}\Vert{\bf 1}_{\Omega_0}\Vert_{BV},\; \forall q\in\ZZ,
\end{equation*}
combined with $L^p=[L^1, L^\infty]_{\frac1p}$, we deduce
\begin{eqnarray*}
\Vert\dot\Delta_q{\bf 1}_{\Omega_0}\Vert_{L^p}&\lesssim& \Vert\dot\Delta_q{\bf 1}_{\Omega_0}\Vert_{L^1}^{\frac1p}\Vert\dot\Delta_q{\bf 1}_{\Omega_0}\Vert_{L^\infty}^{1-\frac1p}\\
&\lesssim& 2^{-\frac{q}{p}}\Vert{\bf 1}_{\Omega_0}\Vert_{BV}\Vert{\bf 1}_{\Omega_0}\Vert_{L^\infty}\\
&\lesssim&2^{-\frac{q}{p}}\Vert{\bf 1}_{\Omega_0}\Vert_{L^\infty\cap BV}.
\end{eqnarray*}
Here we have used the fact that $\dot\Delta_q$ maps continuously $L^\infty$ into it self. Thus we obtain
\begin{equation*}
2^{-\frac{q}{p}}\Vert\dot\Delta_q{\bf 1}_{\Omega_0}\Vert_{L^p}\le C\Vert{\bf 1}_{\Omega_0}\Vert_{L^\infty\cap BV},\quad \forall q\in\ZZ.
\end{equation*}
Taking the supremum over $q\in\ZZ$, we finally obtain that ${\bf 1}_{\Omega_0}\in \dot B^{\frac{1}{p}}_{p,\infty}$ 
\end{proof}
Now, we state the general version of the Theorem \ref{rate}. Roughly speaking we have:  
\begin{Theo}\label{rateconvergence} Let $(v_{\mu}, \rho_\mu)$ and  $(v, \rho)$ be the solutions of (B$_{\mu}$) and (B$_{0}$) respectively with $(v_{\mu}^0, \rho_{\mu}^0)$ and $(v^0, \rho^0)$ their initial data. Let $\omega^0_{\mu},\; \omega^0$ be their vorticties with $\omega_{\mu}^0\in L^\infty\cap B^{\frac1p}_{p,\infty},\; \omega^0\in L^1\cap L^\infty$ and $\rho^{0}, \rho_{\mu}^0\in L^1\cap\L^p$. Setting $\Pi(t)=\Vert v_{\mu}-v\Vert_{L^{p}}+\Vert \rho_{\mu}-\rho\Vert_{L^{p}}$, then we have the following rate of convergence.
\begin{equation*}
\Pi(t)\le  C e^{C(t+V_\mu(t)+V(t))}\Big(\Pi(0)+(\mu t)^{\frac12+\frac{1}{2p}}(1+\mu t)\big(\|\omega_{\mu}^{0}\|_{B^{\frac1p}_{p,\infty}}+\Vert\rho_\mu^{0}\Vert_{L^p}\big)\Big)\quad p\in[2,\infty[,
\end{equation*}
where $$V_{\mu}(t)=\int_{0}^{t}\|\nabla v_{\mu}(\tau)\|_{L^\infty}d\tau,\quad V(t)=\int_{0}^{t}\|\nabla v(\tau)\|_{L^\infty}d\tau.$$
\end{Theo}
The proof of the preceeding Theorem needed the following interpolation result.

\begin{prop}\label{Intp} Let $(p,r,\eta)\in[1, \infty]^2\times]-1,1[$ and $v_{\mu}$ be a free divergence vector field depicted by the Biot-Savart law $v_\mu=\Delta^{-1}\nabla^\perp\omega_\mu$, i.e.,
\begin{equation*}
v_{\mu}(t,x)=\frac{1}{2\pi}\int_{\RR^2}\frac{(x-y)^\perp}{\vert x-y\vert^2}\omega_{\mu}(t,y)dy.
\end{equation*}
Then the following estimate holds true.
\begin{equation*}
\Vert \Delta v_{\mu}\Vert_{L^{r}_t L^p}\le C\Vert\omega_{\mu}\Vert_{\widetilde{L}_t^{r} B_{p, \infty}^{\eta}}^{\frac{1+\eta}{2}}\Vert \omega_{\mu}\Vert_{\widetilde{L}_t^{r} B_{p,\infty}^{2+\eta}}^{\frac{1-\eta}{2}}.
\end{equation*}
\end{prop} 
\begin{proof} To prove this estimate, let $N\in\NN$ be a parameter that will be chosen judiciously later. The fact $\Delta v_{\mu}=\nabla^{\perp}\omega_{\mu}$,  interpolation in frequency and Bernstein inequality enable us to write

\begin{eqnarray}\label{Ayat11}
\|\Delta v_{\mu}\|_{L^{r}_{t} L^p}&\le& \sum_{q\le N}\Vert\Delta_q\nabla^{\perp}\omega_{\mu}\Vert_{L^{r}_t L^p}+\sum_{q> N}\Vert\Delta_q\nabla^{\perp}\omega_{\mu}\Vert_{L^{r}_t L^p}\\
\nonumber&\le &\sum_{q\le N}2^{q(1-\eta)}2^{q\eta}\Vert\Delta_q\omega_{\mu}\Vert_{L^{r}_{t} L^p}+\sum_{q>N}2^{q(-1-\eta)}2^{q(2+\eta)}\Vert\Delta_{q}\omega_{\mu}\Vert_{L^{r}_{t} L^p}\\
\nonumber&\le & 2^{N(1-\eta)}\Vert\omega_{\mu}\Vert_{\widetilde L^{r}_t B^{\eta}_{p,\infty}}+2^{-N(1+\eta)}\Vert\omega_{\mu}\Vert_{\widetilde L^{r}_t B^{2+\eta}_{p,\infty}}.
\end{eqnarray}
Now, we choose $N$ such that
\begin{equation*}
2^{N(1-\eta)}\Vert\omega_{\mu}\Vert_{\widetilde{L}^{r}_t B^{\eta}_{p,\infty}}\approx 2^{-N(1+\eta)}\Vert\omega_{\mu}\Vert_{\widetilde{L}^{r}_t B^{2+\eta}_{p,\infty}},
\end{equation*}
whence
\begin{equation}\label{Ayat111}
2^{2N}\approx\frac{\Vert\omega_{\mu}\Vert_{\widetilde{L}^{r}_t B^{2+\eta}_{p,\infty}}}{\Vert\omega_{\mu}\Vert_{\widetilde{L}^{r}_t B^{\eta}_{p,\infty}}}.
\end{equation}
Inserting \eqref{Ayat111} in \eqref{Ayat11}, we obtain the desired estimate and so the proof is completed.
\end{proof}

\begin{proof}[Proof of the Theorem \ref{rateconvergence}] We set $U=v_{\mu }-v, \Theta=\rho _{\mu }-\rho$ and $P=p _{\mu }-p$. We intend to estimate the quantity $\Vert U\Vert_{L^p}+\Vert\Theta\Vert_{L^p}$. To do this, making a few computations we discover that  $U$ and $\Theta$ evolve the nonlinear equations,
\begin{equation}\label{princ-syst1}
\left\{ 
\begin{array}{ll}
\partial _{t}U+( v_{\mu }\cdot \nabla) U-\mu \Delta v_{\mu}+\nabla P=\Theta e_{2}-(U\cdot\nabla)v & \textrm{ $(t,x)\in 
\mathbb{R}_{+}\times \mathbb{R}^{2}$},\\ 
\partial _{t}\Theta+(v_{\mu }\cdot \nabla)\Theta -\Delta\Theta=-U\cdot\nabla \rho & \textrm{ $(t,x)\in 
\mathbb{R}_{+}\times \mathbb{R}^{2}$},\\ 
\Div U=0, &  \\ 
U_{|t=0}=U_0,\quad \Theta_{|t=0}=\Theta_0. 
\end{array}
\right.   \tag{$\widetilde{B}_{\mu}$}
\end{equation}

$\bullet$ {\bf Estimate of $\Vert U(t)\Vert_{L^p}$.} Multiply the first equation in ($\widetilde{B}_{\mu}$)  by $U\vert U\vert^{p-2}$, and integrating by parts over the space variables $\RR^2$, then in view of $\Div v_\mu=\Div v=0$, it follows
\begin{eqnarray*}
\frac{1}{p}\frac{d}{dt}\Vert U(t)\Vert_{L^{p}}^{p}&\le&\int_{\RR^2}\big\vert\nabla P\cdot U\vert U\vert^{p-2}\big\vert dx+\mu\int_{\RR^2}\big\vert\Delta v_{\mu}\cdot U\vert U\vert^{p-2}\big\vert dx+\int_{\RR^2}\big\vert (U\cdot\nabla)v\cdot U\vert U\vert^{p-2}\big\vert dx\\
&&+\int_{\RR^2}\big\vert\Theta e_2\cdot U\vert U\vert^{p-2}\big\vert dx.
\end{eqnarray*}
H\"older's inequality yields
\begin{equation*}
\frac{1}{p}\frac{d}{dt}\Vert U(t)\Vert_{L^{p}}^{p}\le\Vert\nabla P(t)\Vert_{L^p}\Vert U(t)\Vert_{L^p}^{p-1}+\mu\Vert\Delta v_{\mu}(t)\Vert_{L^p}\Vert U(t)\Vert_{L^p}^{p-1}+\Vert\nabla v\Vert_{L^\infty}\Vert U(t)\Vert_{L^p}^{p}+\Vert\Theta(t)\Vert_{L^p}\Vert U(t)\Vert_{L^p}^{p-1},
\end{equation*}
so, integrating in time over $[0, t]$, we obtain
\begin{eqnarray}\label{Ayat00}
\Vert U(t)\Vert_{L^{p}}&\le &\Vert U_0\Vert_{L^{p}}+\int_0^t\Vert\nabla P(\tau)\Vert_{L^p}d\tau+\mu\int_0^t\Vert\Delta v_{\mu}(\tau)\Vert_{L^p}d\tau\\
\nonumber &&+\int_0^t\Vert\nabla v(\tau)\Vert_{L^\infty}\Vert U(\tau)\Vert_{L^p}d\tau+\int_0^t\Vert\Theta(\tau)\Vert_{L^p}d\tau.
\end{eqnarray}
Concerning the term $\Vert\nabla P\Vert_{L^p}$, applying the "$\Div$" operator to the first equation of ($\widetilde{B}_{\mu}$), one finds after easy algebraic computations
\begin{eqnarray*}
-\Delta P=\Div\big(U\cdot\nabla(v_{\mu}+v)\big)+\partial_2\Theta,
\end{eqnarray*}
then we have
\begin{equation*}
-\nabla P=\nabla\Delta^{-1}\Div\big(U\cdot\nabla(v_{\mu}+v)\big)+\nabla\Delta^{-1}\partial_2\Theta.
\end{equation*}
The boundedness of Riesz transform on $L^p,\; p\in]1,\infty[$ into it self leading to
\begin{equation*}
\Vert\nabla P\Vert_{L^p}\lesssim\Vert U\Vert_{L^p}\big(\Vert\nabla v_{\mu}\Vert_{L^\infty}+\Vert\nabla v\Vert_{L^\infty}\big)+\Vert\Theta\Vert_{L^p}.
\end{equation*}
Inserting the above estimate into \eqref{Ayat00}, we deduce that
\begin{eqnarray}\label{Ayat000}
\Vert U(t)\Vert_{L^{p}}&\lesssim &\Vert U_0\Vert_{L^{p}}+\int_0^t\Vert U(\tau)\Vert_{L^p}\big(\Vert\nabla v_{\mu}(\tau)\Vert_{L^\infty}+\Vert\nabla v(\tau)\Vert_{L^\infty}\big)d\tau\\
\nonumber&&+\mu\int_0^t\Vert\Delta v_{\mu}(\tau)\Vert_{L^p}d\tau+\int_0^t\Vert\Theta(\tau)\Vert_{L^p}d\tau.
\end{eqnarray}

$\bullet$ {\bf Estimate of $\|\Theta\|_{L^1_t L^p}$}. Multiplying the second equation in ($\widetilde{B}_{\mu}$) by $\Theta\vert\Theta\vert^{p-2}$, and integrating by parts over $\RR^2$. Then by virtue of $\Div v_\mu=\Div v=0$, it happens
\begin{equation*}
\frac{1}{p}\frac{d}{dt}\Vert \Theta(t)\Vert_{L^{p}}^{p}+(p-1)\int_{\RR^2}\vert\nabla\Theta(t)\vert^2\vert\Theta(t)\vert^{p-2}dx\le \int_{\RR^2}\big\vert U(t)\cdot\nabla\rho(t)\vert\vert\Theta(t)\vert^{p-1} dx,
\end{equation*}
owing to H\"older's inequality, we shall have
\begin{equation*}
\frac{1}{p}\frac{d}{dt}\Vert \Theta(t)\Vert_{L^{p}}^{p}+(p-1)\int_{\RR^2}\vert\nabla\Theta(t)\vert^2\vert\Theta(t)\vert^{p-2}dx\le \Vert U(t)\Vert_{L^p}\Vert\nabla\rho(t)\Vert_{L^\infty}\Vert\Theta(t)\Vert_{L^p}^{p-1}.
\end{equation*}
Since the second term of the left-hand side has a non-negative sign, one obtains
\begin{equation*}
\frac{d}{dt}\Vert \Theta(t)\Vert_{L^{p}}\lesssim \Vert U(t)\Vert_{L^p}\Vert\nabla\rho(t)\Vert_{L^\infty}.
\end{equation*}
Integrating in time over $[0, t]$, we get
\begin{equation}\label{Ayat1}
\Vert \Theta(t)\Vert_{L^{p}}\lesssim \Vert \Theta_0\Vert_{L^{p}}+\int_0^t\Vert U(\tau)\Vert_{L^p}\Vert\nabla\rho(\tau)\Vert_{L^\infty}d\tau.
\end{equation}
Putting together \eqref{Ayat000} and \eqref{Ayat1}, we readily get
\begin{eqnarray*}
\nonumber\Vert U(t)\Vert_{L^{p}}+\Vert \Theta(t)\Vert_{L^{p}}&\lesssim& \Vert U_0\Vert_{L^{p}}+\Vert \Theta_0\Vert_{L^{p}}+\int_0^t\Vert U(\tau)\Vert_{L^p}\big(\Vert\nabla v_{\mu}(\tau)\Vert_{L^\infty}+\Vert\nabla v(\tau)\Vert_{L^\infty}\big)d\tau\\
\nonumber&&+\mu\int_0^t\Vert\Delta v_{\mu}(\tau)\Vert_{L^p}d\tau+\int_0^t\Vert\Theta(\tau)\Vert_{L^p}d\tau+\int_0^t\Vert U(\tau)\Vert_{L^p}\Vert\nabla\rho(\tau)\Vert_{L^\infty}d\tau.
\end{eqnarray*}
Since $\Pi(t)\triangleq\Vert U(t)\Vert_{L^{p}}+\Vert \Theta(t)\Vert_{L^{p}}$, then after few caculations we find that
\begin{equation*}
\Pi(t)\lesssim \Pi(0)+\int_0^t\big(1+\Vert\nabla v_{\mu}(\tau)\Vert_{L^\infty}+\Vert\nabla v(\tau)\Vert_{L^\infty}+\|\nabla\rho(\tau)\|_{L^\infty}\big)\Pi(\tau)d\tau+\mu\int_0^t\Vert\Delta v_{\mu}(\tau)\Vert_{L^p}d\tau.
\end{equation*}
Using Gronwall's inequality, we can write
\begin{eqnarray}\label{Rate00}
\Pi(t)&\lesssim& e^{Ct}e^{V_{\mu}(t)+V(t)+\Vert\nabla\rho\Vert_{L^1_t L^\infty}}\big(\Pi(0)+\mu\Vert\Delta v_{\mu}\Vert_{L^1_t L^p}\big).
\end{eqnarray}
We now turn to the estimate of the principal term $\mu\Vert\Delta v_{\mu}\Vert_{L^1_t L^p}$ which provides the desired rate of convergence. For this aim, we apply the Proposition \ref{Intp} by taking $\eta=\frac1p$ and $r=1$,

\begin{equation}\label{Ayat2}
\mu\Vert \Delta v_{\mu}\Vert_{L_t^1 L^p}\le\underbrace{\mu\Vert\omega_{\mu}\Vert_{\widetilde{L}_t^{1} B_{p, \infty}^{\frac1p}}^{\frac{1}{2}+\frac{1}{2p}}}_{\textnormal{I}}\underbrace{\Vert \omega_{\mu}\Vert_{\widetilde{L}_t^{1} B_{p,\infty}^{2+\frac1p}}^{\frac{1}{2}-\frac{1}{2p}}}_{\textnormal{II}}.
\end{equation}
For the term I, applying the H\"older inequality in the time variable, we deduce that
\begin{equation*}
\textnormal{I}\le\mu t^{\frac{1}{2}+\frac{1}{2p}}\Vert\omega_{\mu}\Vert^{\frac{1}{2}+\frac{1}{2p}}_{\widetilde{L}_{t}^{\infty} B^{\frac1p}_{p,\infty}}.
\end{equation*}
Put $r=\infty, s=\frac1p, p_1=p, p_2=\infty$, Proposition \ref{Persistance} tell us
\begin{equation*}\label{Ayat3}
\Vert\omega_{\mu}\Vert_{\widetilde{L}_t^\infty B^{\frac1p}_{p,\infty}}\le Ce^{CV_{\mu}(t)}\Big(\Vert\omega^0_{\mu}\Vert_{B^{\frac1p}_{p,\infty}}+\Vert\nabla\rho_{\mu}\Vert_{L^1_t B^{\frac1p}_{p,\infty}}\Big).
\end{equation*}
Hence
\begin{equation*}
\textnormal{I}\le Ce^{CV_{\mu}(t)}\mu t^{\frac{1}{2}+\frac{1}{2p}}\Big(\Vert\omega^0_{\mu}\Vert_{B^{\frac1p}_{p,\infty}}+\Vert\nabla\rho_{\mu}\Vert_{L^1_t B^{\frac1p}_{p,\infty}}\Big)^{\frac{1}{2}+\frac{1}{2p}}.
\end{equation*}
Concerning II, a new use of the Proposition \ref{Persistance} gives for $r=1, s=\frac1p, p_1=p, p_2=\infty$ the following
\begin{equation*}
\Vert\omega_{\mu}\Vert_{\tilde L_t^1 B^{2+\frac1p}_{p,\infty}}\le Ce^{CV_{\mu}(t)}\mu^{-1}(1+\mu t)\Big(\Vert\omega^0_{\mu}\Vert_{B^{\frac1p}_{p,\infty}}+\Vert\nabla\rho_{\mu}\Vert_{L^1_t B^{\frac1p}_{p,\infty}}\Big).
\end{equation*}
Accordingly, we infer
\begin{equation*}\label{Ayat4}
\textnormal{II}\le Ce^{CV_{\mu}(t)}\mu^{\frac{1}{2p}-\frac{1}{2}}(1+\mu t)^{\frac{1}{2}-\frac{1}{2p}}\Big(\Vert\omega^0_{\mu}\Vert_{B^{\frac1p}_{p,\infty}}+\Vert\nabla\rho_{\mu}\Vert_{L^1_t B^{\frac1p}_{p,\infty}}\Big)^{\frac{1}{2}-\frac{1}{2p}}.
\end{equation*}
Combining I and II, \eqref{Ayat2} becomes
\begin{equation*}
\mu\Vert\Delta v_{\mu}\Vert_{L_t^1 L^p}\lesssim Ce^{CV_{\mu}(t)}(\mu t)^{\frac{1}{2}+\frac{1}{2p}}(1+\mu t)^{\frac{1}{2}-\frac{1}{2p}}\Big(\Vert\omega^0_{\mu}\Vert_{B^{\frac1p}_{p,\infty}}+\Vert\nabla\rho_{\mu}\Vert_{L^1_t B^{\frac1p}_{p,\infty}}\Big).
\end{equation*}
By means of the embedding $\widetilde L^1_t B^{\frac1p}_{p,1}= L^1_t B^{\frac1p}_{p,1}\hookrightarrow L^1_t B^{\frac1p}_{p,\infty}$, the last estimate becomes
\begin{equation*}\label{Ayat5}
\mu\Vert\Delta v_{\mu}\Vert_{L_t^1 L^p}\lesssim Ce^{CV_{\mu}(t)}(\mu t)^{\frac{1}{2}+\frac{1}{2p}}(1+\mu t)^{\frac{1}{2}-\frac{1}{2p}}\Big(\Vert\omega^0_{\mu}\Vert_{B^{\frac1p}_{p,\infty}}+\Vert\nabla\rho_{\mu}\Vert_{\widetilde L^1_t B^{\frac1p}_{p,1}}\Big)
\end{equation*}
together with \eqref{Rate00}, one obtains
\begin{eqnarray*}\label{Rate0022}
\nonumber\Pi(t)&\lesssim& Ce^{C(t+V_{\mu}(t)+V(t)+\Vert\nabla\rho\Vert_{L^1_t L^\infty})}\bigg(\Pi(0)+(\mu t)^{\frac{1}{2}+\frac{1}{2p}}(1+\mu t)^{\frac{1}{2}-\frac{1}{2p}}\Big(\Vert\omega^0_{\mu}\Vert_{B^{\frac1p}_{p,\infty}}+\Vert\nabla\rho_{\mu}\Vert_{\widetilde L^1_t B^{\frac1p}_{p,1}}\Big)\bigg).
\end{eqnarray*}
For the term $\Vert\nabla\rho\Vert_{L^1_t L^\infty}$, applying the Propositon \ref{propasyy1} the last estimate takes the form
\begin{equation}\label{Rate0022}
\Pi(t)\lesssim Ce^{C(t+V_{\mu}(t)+V(t))}\bigg(\Pi(0)+(\mu t)^{\frac{1}{2}+\frac{1}{2p}}(1+\mu t)^{\frac{1}{2}-\frac{1}{2p}}\Big(\Vert\omega^0_{\mu}\Vert_{B^{\frac1p}_{p,\infty}}+\Vert\nabla\rho_{\mu}\Vert_{\widetilde L^1_t B^{\frac1p}_{p,1}}\Big)\bigg).
\end{equation}

To end the proof of our claim, let us estimate $\Vert\nabla\rho_{\mu}\Vert_{\widetilde L^1_t B^{\frac1p}_{p,1}}$. Note that $\nabla$ maps continuously $B^{1+\frac1p}_{p,1}$ into $B^{\frac1p}_{p,1}$, then the Proposition \ref{Persistance} combined with $L^p\hookrightarrow B^{\frac1p-1}_{p,1}$  gives for $p>1$
\begin{eqnarray}\label{Toto1}
\Vert\rho_{\mu}\Vert_{\widetilde L_t^1 B_{p,1}^{\frac1p+1}}&\le & Ce^{CV_\mu(t)}(1+t)\|\rho^0_\mu\|_{B_{p,1}^{\frac{1}{p}-1}}\\
\nonumber&\le & Ce^{CV_\mu(t)}(1+t)\|\rho^0_\mu\|_{L^p}.
\end{eqnarray}
Plugging \eqref{Toto1} in \eqref{Rate0022}, we find that
\begin{equation}\label{Rate0022222}
\Pi(t)\lesssim Ce^{C(t+V_{\mu}(t)+V(t))}\bigg(\Pi(0)+(\mu t)^{\frac{1}{2}+\frac{1}{2p}}(1+\mu t)^{\frac{1}{2}-\frac{1}{2p}}\Big(\Vert\omega^0_{\mu}\Vert_{B^{\frac1p}_{p,\infty}}+\Vert\rho_{\mu}^0\Vert_{L^p}\Big)\bigg).
\end{equation}
Hence the proof of the Theorem \ref{rateconvergence} is accomplished.
\end{proof}

\subsection{Proof of Theorem \ref{rate}} 
{\bf (i)} Substituting \eqref{H} and \eqref{M} into \eqref{Rate0022222} and the fact ${\bf 1}_{\Omega_0}\in B^{\frac1p}_{p,\infty}$, it happens for $\mu\in]0,1[$ 
\begin{equation*}\label{Rate-Convergence}
\Pi(t)\lesssim C_0e^{e ^{C_0 t\log^2(1+t)}}(\mu t)^{\frac{1}{2}+\frac{1}{2p}}.
\end{equation*}

{\bf (ii)} To estimate $\omega_{\mu}-\omega$ in $L^p-$norm, using the definition of $\omega_\mu$ and $\omega$ we shall have
\begin{equation*}
\|\omega_{\mu}(t)-\omega(t)\|_{L^p}\le \|\nabla (v_{\mu}(t)-v(t))\|_{L^p}
\end{equation*}
combined with $B^0_{p,1}\hookrightarrow L^p$ and Bernstein inequality, we readily get
\begin{equation}\label{OAAA}
\|\omega_{\mu}(t)-\omega(t)\|_{L^p}\lesssim \| v_{\mu}(t)-v(t)\|_{B^{1}_{p,1}}.
\end{equation}

On the other hand, let $N$ be a fixed number that will be chosen later. Again Bernstein's inequality leading to
\begin{eqnarray}\label{OAAA2}
\| v_{\mu}(t)-v(t)\|_{B^{1}_{p,1}}&\le&\sum_{q\le N}2^{q}\| \Delta_q (v_{\mu}(t)-v(t))\|_{L^{p}}+\sum_{q>N}2^{-\frac{q}{p}}2^{\frac{q}{p}}\| \Delta_q\nabla (v_{\mu}(t)-v(t))\|_{L^{p}}\\
\nonumber &\lesssim& 2^N \| v_{\mu}(t)-v(t)\|_{L^p}+\sup_{q\ge-1}2^{\frac{q}{p}} \| \nabla(v_{\mu}(t)-v(t))\|_{L^p}\sum_{q>N}2^{-\frac{q}{p}}\\
\nonumber &\lesssim& 2^N \| v_{\mu}(t)-v(t)\|_{L^p}+2^{-\frac{N}{p}} \|\nabla( v_{\mu}(t)-v(t))\|_{B^{\frac1p}_{p,\infty}}.
\end{eqnarray}
Taking
\begin{equation*}
2^{N(1+\frac1p)}\approx\frac{\|\nabla( v_{\mu}(t)-v(t))\|_{B^{\frac1p}_{p,\infty}}}{\| v_{\mu}(t)-v(t)\|_{L^p}}.
\end{equation*}

By means of Cald\'eron-Zygmund inequality of the Proposition \ref{CZyg}, \eqref{OAAA} and \eqref{OAAA2}, it follows
\begin{equation*}
\| v_{\mu}(t)-v(t)\|_{B^{1}_{p,1}}\lesssim \| v_{\mu}(t)-v(t)\|_{L^p}^{\frac{1}{p+1}}\|\omega_{\mu}(t)-\omega(t)\|_{B^{\frac1p}_{p,\infty}},
\end{equation*}
whence 
\begin{equation*}
\|\omega_{\mu}(t)-\omega(t)\|_{L^p}\le \| v_{\mu}(t)-v(t)\|_{L^p}^{\frac{1}{p+1}}\|\omega_{\mu}(t)-\omega(t)\|_{B^{\frac1p}_{p,\infty}}
\end{equation*}
in accordance with the Theorem \ref{rate}, it holds
\begin{equation*}\label{OAAA1}
\|\omega_{\mu}(t)-\omega(t)\|_{L^p}\le C_0 e^{e^{C_0 t\log^{2}(2+t)}}(\mu t)^{\frac{1}{2p}}(1+\mu t)\|\omega_{\mu}(t)-\omega(t)\|_{B^{\frac1p}_{p,\infty}}.
\end{equation*}
To finalize, let us estimate $\|\omega_{\mu}(t)-\omega(t)\|_{B^{\frac1p}_{p,\infty}}$. To do this, using the persistence of Besov spaces explicitly formulated in the Proposition \ref{prop1}, one gets
\begin{eqnarray*}
\|\omega_{\mu}(t)-\omega(t)\|_{B^{\frac1p}_{p,\infty}}&\le& \|\omega_{\mu}(t)\|_{B^{\frac1p}_{p,\infty}}+\|\omega(t)\|_{B^{\frac1p}_{p,\infty}}\\
\nonumber &\le& Ce^{C(V_{\mu}(t)+V(t))}\Big(\|\omega_{\mu}^{0}\|_{B^{\frac1p}_{p,\infty}}+\|\omega^{0}\|_{B^{\frac1p}_{p,\infty}}+\|\nabla\rho_{\mu}\|_{L^1_t B^{\frac1p}_{p,\infty}}+\|\nabla\rho\|_{L^1_t B^{\frac1p}_{p,\infty}}\Big).
\end{eqnarray*}
The last two terms of the right-hand side stem from similar arguments as in \eqref{Toto1}. Then thanks to \eqref{H} and \eqref{M}, we end up with
\begin{equation*}
\|\omega_{\mu}(t)-\omega(t)\|_{B^{\frac1p}_{p,\infty}}\le C_{0}e^{e^{C_0t\log^2(2+t)}}(\mu t)^{\frac{1}{2p}}(1+\mu t).
\end{equation*}
This achieves the proof of the aimed estimate.   

\subsection{Optimality of the rate of convergence}\label{optimality} In this paragraph we shall give the  proof of Theorem \ref{optimality0} by  showing that $(\mu t)^{\frac{1}{2p}}$ is optimal  in $L^p$ norm in the case of a circular vortex patch and $\rho_{\mu}^0$ and $\rho^0$ are constant densities. 
\begin{proof}[Proof of  Theorem \ref{optimality0}] Since the initial data  $\omega_{\mu}^0=\omega^0={\bf 1}_{\mathbb{D}}$ are radial then this structure is preserved in the evolution and thus
\begin{equation*}
v_\mu\cdot\nabla\omega_\mu,\quad v\cdot\nabla\omega=0.
\end{equation*}
Therefore the equation of $\omega_{\mu}$ (resp. $\omega$) takes the following form
\begin{equation*}
\partial_t\omega_{\mu}-\mu\Delta\omega_{\mu}=0,\quad \partial_t\omega=0.
\end{equation*}
Recall that the solutions of the above equations are given by
\begin{equation}\label{solution}
\omega_{\mu}(t,x)=K_{\mu t}\star \omega^0_{\mu}(x),\quad \omega(t,x)=\omega^0(x),
\end{equation}
where $K_t$ is the heat kernel defined by
\begin{equation*}
K_t(x)\triangleq\frac{1}{4\pi t}e^{-\frac{\vert x\vert^2}{4t}}
\end{equation*}
and satisfies
\begin{equation*}
\int_{\RR^2}K_t(x)dx=1.
\end{equation*} 
On the other hand, setting $W(t,x)=\omega_{\mu}(t,x)-\omega(t,x)$. Then in view of the requirement \eqref{solution}, we have
\begin{equation*}
W(t,x)=\int_{\RR^2}K_{\mu t}(x-y)[{\bf 1}_{\mathbb{D}}(y)-{\bf 1}_{\mathbb{D}}(x)]dy.
\end{equation*}
For $\vert x\vert<1$ we have
\begin{eqnarray*}
W(t,x)&=&\int_{\{\vert y\vert\ge1\}}K_{\mu t}(x-y)dy\\
&=&\frac{1}{4\pi\mu t}\int_{\{\vert y\vert\ge1\}}e^{-\frac{\vert x-y\vert^2}{4\mu t}}dy.
\end{eqnarray*}
Introduce $Z(t,x)=W(t,\sqrt{\mu t}x)$ and make the change of variables $y=\sqrt{\mu t}z$, one gets
\begin{equation}\label{formula}
Z(t,x)=\frac{1}{4\pi}\int_{\{\vert z\vert\ge\frac{1}{\sqrt{\mu t}}\}}e^{-\frac{\vert x-z\vert^2}{4}}dz,\quad\vert x\vert\le\frac{1}{\sqrt{\mu t}}.
\end{equation}
Let $\mu t\le1,$ then
\begin{eqnarray}\label{crucial0}
\|W(t)\|_{L^p(\RR^2)}&\ge&\|W(t)\|_{L^p(1-\sqrt{\mu t}\le\vert x\vert\le 1)}\\
\nonumber&\ge& (\mu t)^{\frac{1}{p}}\|Z(t)\|_{L^p(\frac{1}{\sqrt{\mu t}}-1\le\vert x\vert\le\frac{1}{\sqrt{\mu t}})}.
\end{eqnarray}
Now, our task is to prove the following requirement
\begin{equation}\label{crucial}
\|Z(t)\|_{L^p(\frac{1}{\sqrt{\mu t}}-1\le \vert x\vert\le\frac{1}{\sqrt{\mu t}})}\ge C_2(\mu t)^{-\frac{1}{2p}}.
\end{equation}
For this purpose, we plug the identity $\vert x-z\vert^2\triangleq\vert x\vert^2+\vert z\vert^2-2\langle x,z\rangle$ into \eqref{formula}, 
\begin{equation*}
Z(t,x)=\frac{1}{4\pi}e^{-\frac{\vert x\vert^2}{4}}\int_{\{\vert z\vert\ge\frac{1}{\sqrt{\mu t}}\}}e^{-\frac{\vert z\vert^2}{4}+\frac12\langle x,z\rangle}dz.
\end{equation*}
By rotation invariance, the above equation becomes
\begin{eqnarray*}
Z(t,x)&=&\frac{1}{4\pi}e^{-\frac{\vert x\vert^2}{4}}\int_{0}^{2\pi}\int_{\frac{1}{\sqrt{\mu t}}}^{+\infty}e^{-\frac{r^2}{4}+\frac12 r\vert x\vert\cos\theta}rdrd\theta\\
\nonumber&\ge&\frac{1}{4\pi}e^{-\frac{\vert x\vert^2}{4}}\int_{0}^{\frac{\pi}{2}}\int_{\frac{1}{\sqrt{\mu t}}}^{+\infty}e^{-\frac{r^2}{4}+\frac12 r\vert x\vert\cos\theta}rdrd\theta.
\end{eqnarray*}
Since $\cos\theta\ge 1-\frac{\theta^2}{2}$ for $\theta\ge0$, then we find
\begin{eqnarray*}
\vert Z(t,x)\vert &\ge& \frac{1}{4\pi}\int_{\frac{1}{\sqrt{\mu t}}}^{+\infty}e^{-\frac{\vert x\vert^2}{4}-\frac{r^2}{4}+\frac{r\vert x\vert}{2}}\bigg(\int_{0}^{\frac{\pi}{2}}e^{-\frac{1}{4}r\vert x\vert\theta^2}d\theta\bigg) rdr\\
&=&\frac{1}{4\pi}\int_{\frac{1}{\sqrt{\mu t}}}^{+\infty}e^{-\frac14(\vert x\vert-r)^2}\bigg(\int_{0}^{\frac{\pi}{2}}e^{-\frac{1}{4}r\vert x\vert\theta^2}d\theta\bigg) rdr.
\end{eqnarray*}
Here, we have used Fubini's theorem. For the second integral of the right-hand side, using the change of variables $\alpha=\frac12\sqrt{r \vert x\vert}\theta$, we get

\begin{equation}\label{formula30}
\vert Z(t,x)\vert \ge \frac{1}{2\pi}\int_{\frac{1}{\sqrt{\mu t}}}^{\frac{2}{\sqrt{\mu t}}}e^{-\frac14(\vert x\vert-r)^2}\Bigg(\int_{0}^{\sqrt{r\vert x\vert}\frac{\pi}{4}}e^{-\alpha^2}\frac{d\alpha}{\sqrt{r\vert x\vert}}\Bigg) rdr.
\end{equation}
Since $r\vert x\vert\ge\frac{1}{\sqrt{\mu t}}\Big(\frac{1}{\sqrt{\mu t}}-1\Big)\approx\frac{1}{\mu t}\ge 1$, then we obtain that
$$
\int_{0}^{\sqrt{r\vert x\vert}\frac{\pi}{4}}e^{-\alpha^2}\frac{d\alpha}{\sqrt{r\vert x\vert}}\ge\int_{0}^{\frac{\pi}{4}}e^{-\alpha^2}d\alpha=c.
$$
Consequently for $\frac{1}{\sqrt{\mu t}}-1\le\vert x\vert\le\frac{1}{\sqrt{\mu t}}$, the formula \eqref{formula30} takes the following form
\begin{equation*}
\vert Z(t,x)\vert \ge C\int_{\frac{1}{\sqrt{\mu t}}}^{\frac{2}{\sqrt{\mu t}}}e^{-\frac14(\vert x\vert-r)^2}\sqrt{\frac{r}{\vert x\vert}}dr.
\end{equation*}
But, $\frac{r}{\vert x\vert}\ge\frac{1}{\sqrt{\mu t}}\sqrt{\mu t}=1$ and hence
\begin{equation*}
\vert Z(t,x)\vert \ge C\int_{\frac{1}{\sqrt{\mu t}}}^{\frac{2}{\sqrt{\mu t}}}e^{-\frac14(\vert x\vert-r)^2}dr.
\end{equation*}
Making the change of variables $k=r-\vert x\vert$, we readily get

\begin{equation*}
\vert Z(t,x)\vert \ge C\int_{\frac{1}{\sqrt{\mu t}}-\vert x\vert}^{\frac{2}{\sqrt{\mu t}}-\vert x\vert}e^{-\frac14 k^2}dk.
\end{equation*}
However, $\frac{1}{\sqrt{\mu t}}-\vert x\vert\le1$ and $\frac{2}{\sqrt{\mu t}}-\vert x\vert\ge\frac{1}{\sqrt{\mu t}}$. This leads to
\begin{equation*}
\vert Z(t,x)\vert \ge C\int_{1}^{\frac{1}{\sqrt{\mu t}}}e^{-\frac14 k^2}dk\ge C>0.
\end{equation*}
Therefore, for $\frac{1}{\sqrt{\mu t}}-1\le \vert x\vert\le\frac{1}{\sqrt{\mu t}}$, it follows
\begin{equation}\label{nn}
\vert Z(t,x)\vert\ge C.
\end{equation}
Taking the $L^p-$norm for \eqref{nn} over the annulus $\frac{1}{\sqrt{\mu t}}-1\le \vert x\vert\le\frac{1}{\sqrt{\mu t}}$, it holds 
\begin{eqnarray*}
\|Z(t)\|_{L^p(\frac{1}{\sqrt{\mu t}}-1\le \vert x\vert\le\frac{1}{\sqrt{\mu t}})}&\ge& C\Bigg[\mathscr{L}\bigg(\frac{1}{\sqrt{\mu t}}-1\le \vert x\vert\le\frac{1}{\sqrt{\mu t}}\bigg)\Bigg]^{\frac1p}\\
\nonumber&\ge& C\bigg[2\pi \bigg(\frac{1}{\sqrt{\mu t}}-1\bigg)\bigg]^{\frac{1}{p}}\\
\nonumber&\ge&\widetilde{C}(\mu t)^{-\frac{1}{2p}},
\end{eqnarray*}
where $\mathscr{L}$  is the Lebesgue measure over $\RR^2$. Hence,

\begin{equation*}
\|Z(t)\|_{L^p(\frac{1}{\sqrt{\mu t}}-1\ge \vert x\vert\le\frac{1}{\sqrt{\mu t}})}\ge C_1(\mu t)^{-\frac{1}{2p}}.
\end{equation*}
This leads to the desired estimate stated in \eqref{crucial}. Combining the last estimate with \eqref{crucial0}, we end up with
\begin{equation*}
\|W(t)\|_{L^p(\RR^2)}\ge C_1 (\mu t)^{\frac{1}{2p}}.
\end{equation*}
Now, the proof is completed.
\end{proof}

\section*{Appendix} This section cares with the detailed proof of two Propositions \ref{An1}, \ref{injec} which are used respectively during the proof of Theorem \ref{theo7} and Proposition \ref{Inj}.
\begin{prop}\label{An1} Let $\varepsilon\in]0, 1[, \rho$ be a smooth function  and $v$ be a smooth divergence-free vector field \mbox{on $\RR^2$} with vorticity $\omega$. Assume that $v\in L^2, \omega\in L^2\cap L^\infty$ and $\rho\in L^2\cap L^p$, with $p>\frac{2}{1-\EE}.$ Then the following statement  holds true,   
$$
\big\Vert\big[\mathcal{L}, v\cdot\nabla\big]\rho\big\Vert_{C^{\varepsilon}}\le C_0.
$$
\end{prop}

\begin{proof}
Recall from \cite{HZ-poche} the following commutator estimate,
\begin{eqnarray}\label{Poc3}
\big\Vert\big[\mathcal{L}, v\cdot\nabla\big]\rho\big\Vert_{C^{\varepsilon}}\lesssim\|v\|_{L^2}\|\rho\|_{L^2}+\Vert\omega\Vert_{L^2\cap L^{\infty}}\Vert\rho\Vert_{L^p},\quad p>\frac{2}{1-\EE}.
\end{eqnarray}
Let us estimate the first term of the right-hand side of \eqref{Poc3}. To do this, we apply the energy estimate for the velocity equation, we shall have
\begin{equation*}
\|v(t)\|_{L^2}\le \|v_0\|_{L^2}+\int_0^t\|\rho(\tau)\|_{L^2}d\tau.
\end{equation*}
A new use of \cite{HZ-poche} gives
\begin{equation*}
(1+t)^{\frac12}\|\rho(t)\|_{L^2}\lesssim\|\rho_0\|_{L^1\cap L^2}
\end{equation*}
thus we obtain
\begin{equation*}
\|v(t)\|_{L^2}\le C_0(1+t)^{\frac12}.
\end{equation*}
Combining the last two estimates, we readily get
\begin{equation}\label{Amina1}
\|v(t)\|_{L^2}\|\rho(t)\|_{L^2}\le C_0.
\end{equation}
A usual interpolation inequality between the Lebesgue spaces yields for $p\in[2,+\infty[$
\begin{eqnarray}\label{Amina2}
\|\rho(t)\|_{L^p}&\le& \|\rho(t)\|_{L^2}^{\frac2p}\|\rho_0\|_{L^\infty}^{1-\frac2p}\\
\nonumber&\le& C_0(1+t)^{-\frac1p}.
\end{eqnarray}
Here we have used the maximum principle for the density equation. Putting together \eqref{Amina1}, \eqref{Amina2} and Proposition \ref{propasyy1}, we finally get
\begin{eqnarray*}
\big\|\big[\mathcal{L},v\cdot\nabla  \big]\rho\big\|_{C^\EE}&\le& C_0+C_0(1+t)^{-\frac1p}\log^{2}(2+t)\\
&\leq& C_0.
\end{eqnarray*}
This completes the proof.
\end{proof}
For the reader's convenience we state the following classical result.
\begin{prop}\label{injec} The following Sobolev embedding is hold.
\begin{equation*}
BV\hookrightarrow \dot B^{1}_{1,\infty}.
\end{equation*}
\end{prop}

\begin{proof} According to \cite{GL,Tr} the equivalent norm to $\dot B^{s}_{p,r}$ is defined for $\ell\in\NN^*,\;0<s<\ell$ and $(p,r)\in[1,\infty]^2$ by
\begin{equation*}
\||u\||_{\dot B^{s}_{p,r}}\triangleq \Bigg(\int_{\RR^N}\vert h\vert^{-sr}\Vert{\bf \Delta}^{\ell}_h f(x)\Vert_{L^p}^r\frac{dh}{\vert h\vert^{N}}\Bigg)^{\frac{1}{r}}.
\end{equation*}
Here the difference operators ${\bf \Delta}^{\ell}_h$ are given by
\begin{equation*}
{\bf \Delta}^1_h={\bf \Delta}_h,\quad {\bf \Delta}_h^{\ell+1}={\bf \Delta}_h\circ{\bf \Delta}^{\ell}_h\quad\forall \ell\in\NN^*,
\end{equation*} 
where ${\bf \Delta}_h$ is defined for every $u\in\mathcal{S}'(\RR^N)$ and $h\in\RR^N$ by $${\bf \Delta}_h u(x)\triangleq u(x+h)-u(x).$$ 
From \eqref{Hom}, we have for $q\in\ZZ$ and $x\in\RR^2$
\begin{equation*}
\dot\Delta_q u(x)=2^{2q}\int_{\RR^2}\mathscr{F}^{-1}\varphi(2^q(x-y))u(y)dy,
\end{equation*}
with $\mathscr{F}^{-1}\varphi$ denotes the inverse Fourier of $\varphi$. As $\varphi(0)=0$ then 
\begin{equation*}
\dot\Delta_q u(x)=2^{2q}\int_{\RR^2}\mathscr{F}^{-1}\varphi(2^q(x-y))\big(u(y)-u(x)\big)dy.
\end{equation*}
 
So, by making a change of variable $z=2^d(x-y)$, we obtain
\begin{eqnarray*}
\dot\Delta_q u(x)&=&2^{2q}\int_{\RR^2}\mathscr{F}^{-1}\varphi(2^q(x-y))(u(y)-u(x))dy\\
\nonumber&=&\int_{\RR^2}\mathscr{F}^{-1}\varphi(z)(u(x-2^{-q}z)-u(x))dz\\
\nonumber&=&\int_{\RR^2}\mathscr{F}^{-1}\varphi(z){\bf \Delta}_{h}u(x)dz,\quad h=-2^{-q}z.
\end{eqnarray*}
Fubini's theorem implies 
\begin{equation}\label{Ok}
\| \dot\Delta_q u\|_{L^1}\le \int_{\RR^2}\mathscr{F}^{-1}\varphi(z)\|{\bf \Delta}_{h}u\|_{L^1}dz.
\end{equation}

We recall from Theorem 13.48 page 415 in \cite{GL}  the following result
\begin{eqnarray*}
\|{\bf \Delta}_{h} u\|_{L^1}&\le &\vert h\vert \vert Du\vert(\RR^2)\\
&=& 2^{-q}\vert z\vert \vert Du\vert(\RR^2).
\end{eqnarray*}
Consequently 
\begin{equation*}
\|{\bf \Delta}_{h} u\|_{L^1}\le 2^{-q}\vert z\vert \Vert u\Vert_{BV}.
\end{equation*}

Inserting the last estimate in \eqref{Ok}, we get for $q\in\ZZ$
\begin{equation*}
\|\dot\Delta_q u(x)\|_{L^1}\le 2^{-q}\|u\|_{BV}\int_{\RR^2}\mathscr{F}^{-1}\varphi(z)\vert z\vert dz.
\end{equation*}
By taking the supremum over $q\in\ZZ$, we obtain the aimed estimate.
\end{proof}
\section*{Acknowledgements}

The authors wish to thank Taoufik  Hmidi from the University of Rennes 1 for the fruitful discussions on the subject during his visit to the University of Batna 2  in May 2016.

\end{document}